\documentclass[10pt]{iopart}
\usepackage{amssymb,setstack,amsthm}
\usepackage{pst-all}
\usepackage{graphicx}
\usepackage{pstricks-add}
\usepackage{fullpage}
\usepackage{subfigure}

\newtheorem{theorem}{Theorem}[section]
\newtheorem{corollary}[theorem]{Corollary}
\newtheorem{lemma}[theorem]{Lemma}
\newtheorem{proposition}[theorem]{Proposition}
\newtheorem{remark}[theorem]{Remark}
\newtheorem*{hypothesis}{Hypothesis $\mathcal{H}$}

\newcommand{\wt}[1]{\widetilde{#1}}
\newcommand{\Div}{\mbox{div}}
\newcommand{\pD}{\partial D}
\newcommand{\pDe}{\partial D_{\ve}}
\newcommand{\ve}{\varepsilon}
\newcommand{\vp}{\varphi}
\newcommand{\D}[2]{\frac{\partial #1}{\partial #2}}
\newcommand{\bR}{\mathbb{R}}
\newcommand{\mK}{K_{I}}
\newcommand{\ld}{\lambda}
\newcommand{\Om}{\Omega}
\newcommand{\om}{\omega}
\newcommand{\La}{\Delta}

\begin{document}
\title[Stable reconstruction of GIBCs]{Stable reconstruction of generalized impedance boundary conditions}
\author{Laurent Bourgeois$^{1}$, Nicolas Chaulet $^{1,2}$, Houssem Haddar $^2$}
\address{$^1$Laboratoire POEMS, 32, Boulevard Victor, 75739 Paris Cedex 15, France}
\address{$^2$INRIA Saclay Ile de France et Ecole Polytechnique, Route de Saclay, 91128 Palaiseau Cedex, France}
\eads{laurent.bourgeois@ensta.fr,nicolas.chaulet@inria.fr, houssem.haddar@inria.fr}

\begin{abstract}
We are interested in  the identification of a Generalized Impedance Boundary Condition  from the far--fields created by one or several incident plane waves at a fixed frequency. We focus on the particular case where this boundary condition is expressed with the help of a second order surface operator: the inverse problem then amounts to retrieve the two functions $\lambda$ and $\mu$ that define this boundary operator.
We first derive a global Lipschitz stability result for the identification of $\ld$ or $\mu$ from the far--field for bounded piecewise constant impedance coefficients and we give a new type of stability estimate when inexact knowledge of the boundary is assumed.
We then introduce an optimization method to identify $\lambda$ and $\mu$, using in particular a $H^1$-type regularization of the gradient.
We lastly show some numerical results in two dimensions, including a study of the impact of some various parameters, and by assuming either an exact knowledge of the shape of the obstacle or an approximate one.
\end{abstract}
\maketitle

\section{Introduction}
We are interested in this work in the identification of boundary coefficients in so--called Generalized Impedance Boundary Conditions (GIBC) on a given obstacle from measurements of the scattered field far from the obstacle associated with one or several incident plane waves at a given frequency. More specifically we shall consider boundary conditions of the type
$$
\D{u}{\nu} + \Div_{\pD} (\mu\nabla_{\pD}u)+\ld u=0 \quad \mbox{ on } \pD
$$
where $\mu$ and $\ld$ are complex valued functions, $\Div_{\pD}$ and $\nabla_{\pD}$ are respectively the surface divergence and the surface gradient on $\pD$ and $\nu$ denotes the outward unit normal on $\pD$. In the case $\mu=0$ this condition is the classical impedance boundary condition (also known as the Leontovitch boundary condition) used for instance to model imperfectly conducting obstacles. The wider class of GIBCs is commonly used to model thin coatings  or gratings as well as more accurate models for imperfectly conducting obstacles (see \cite{BenLem96,DurHadJol06,HadJol02,HadJolNgu05}). Addressing this problem is motivated by applications in non destructive testing, identification problems or modelling related to stealth technology or antennas. 
For instance, one may think of ultrasonic non destructive testing for the TE (transverse electric) polarization of a medium which contains a perfect conductor coated with a thin layer.
In this case a GIBC as presented above is satisfied with $\mu=\delta$ and $\ld=\delta k^2 n$, where $k$ denotes the wave number, $\delta$ is the width of the layer and $n$ is the mean value of the layer index with respect to the normal coordinate.

The use of GIBCs  has at least two advantages for  the inverse problem as compared to the use of an exact model. First, the identification problem becomes less unstable. Second, since solving the forward problem with GIBC is less time consuming, using such model in iterative non--linear methods is more advantageous.

The classical case $\mu=0$ has been addressed in the literature by several authors, from the mathematical point of view in \cite{Sin06,Lab} and from the numerical point of view in \cite{CaCoMo,CaCoMo10}. The problem of recovering both the shape of the obstacle and the impedance coefficient is also considered in \cite{LiuNakSin07,Ser,HeKinSin09,NakSin07}. The case of GIBC has only been recently addressed in \cite{BouHad} where uniqueness and local stability results have been reported. 

The present work complements these first investigations in two directions. The first one is on the theoretical level. Initially we derive a global Lipschitz stability estimate for bounded piecewise constant impedance coefficients $\ld$ and $\mu$. This result is similar to the one obtained in \cite{Sin07} for classical impedances and the Laplace equation in a bounded domain. Other or complementary stability results for the inverse coefficient problem can be found in \cite{AlDeRo03,AleVes05,DiCRon03,Ing97}. The main particularity and difficulty of our analysis are related to the treatment of the second order surface operator appearing in the GIBC. Also in contrast with the work in \cite{Sin07} we make use here of Carleman estimates instead of doubling properties as main tool in deriving the stability estimates. We then prove the stability of the reconstruction of the impedances when only inexact knowledge of the geometry is available. The proof of this result relies on two properties:
\begin{itemize}
\item continuity of the measurements with respect to the obstacle, uniformly with respect to the impedance coefficients,
\item stability for the inverse coefficient problem for a known obstacle. 
\end{itemize}
  Up to our knowledge, this kind of stability result is new. It would be useful for instance when the geometry has been itself reconstructed from measurements using some  qualitative methods (e.g. sampling methods \cite{CakCol,GriKir08}) and therefore is known only approximately. It may  also be useful in understanding the convergence of iterative methods to reconstruct both the obstacle and the coefficients where the updates for the geometry and the physical parameters are made alternatively.
Let us also mention that the proof of our stability result can be straightforwardly extended to other identification problems that enjoy the two properties indicated above.

In a second part, we investigate a numerical optimization method to identify the boundary coefficients. We propose a reconstruction procedure based on a steepest descent method with $H^1(\pD)$ regularization of the gradient. The accuracy and stability of the inversion scheme is tested through various numerical experiments in a $2D$ setting. Special attention is given to the case of non regular coefficients and inexact knowledge of the boundary $\pD$.

The outline of our article is the following. In section $2$ we introduce and study the forward and inverse problems. Section $3$ is dedicated to the derivation of a stability result with inexact geometry. The numerical part is the subject of section $4$.
\section{The forward and inverse problems}
\label{SecDirect}
\subsection{The forward scattering problem}
\noindent Let $D$ be an open bounded domain of $\bR^d$, $d=2$ or $3$ with a Lipschitz continuous boundary $\pD$, $\Om:=\bR^d \setminus \overline{D}$ and $(\ld,\mu)\in (L^{\infty}(\pD))^2$ be the impedance coefficients. The scattering problem with generalized impedance boundary conditions (GIBC) consists in finding $u = u^s+ u^i$ such that
\renewcommand{\arraystretch}{1.5}
\begin{equation}
\label{PbInit}
\cases{\La u + k^2u = 0  \quad \mbox{in} \; \Om, \\
\Div _{\pD} (\mu \nabla_{\pD} u) +\frac{\partial u}{\partial \nu} + \ld u =0  \quad \mbox{on} \; \pD}
\end{equation}
and $u^s$ satisfies the Sommerfeld radiation condition
\begin{equation*}
 \lim \limits_{R \to \infty} \int_{|x|=R}\left|\D{u^s}{r} - iku^s \right|^2ds =0
\end{equation*}
where $k$ is the wave number, $u^i = e^{ik\hat{d}\cdot x}$ is an incident plane wave where $\hat{d} $ belongs to the unit sphere of $\bR^d$ denoted $S^{d-1}$ and  $u^s \in V :=  \{ v \in \mathcal{D}'(\Om) ,\varphi v \, \in \, H^1(\Om) \, \forall \, \varphi \, \in \mathcal{D}(\bR^d) \mbox{  and  }  v_{|\pD} \, \in \,  H^1(\pD) \}$ is the scattered field.  For $v \in H^1(\pD)$ the surface gradient $\nabla_{\pD}v$ lies in  $L^2_t(\pD):=\{V \in (L^2(\pD))^d\, ,\, V\cdot\nu=0\}$. Moreover, $\Div _{\pD} (\mu \nabla_{\pD} u)$ is defined in $H^{-1}(\pD)$ for $\mu \in L^{\infty}(\pD)$ by
\[
 \fl \langle \Div _{\pD} (\mu \nabla_{\pD} u), v \rangle_{H^{-1}(\pD),H^{1}(\pD)} := - \int_{\pD} \mu \nabla_{\pD}u \cdot\nabla_{\pD}v ds \quad \forall v \in H^{1}(\pD) \,.
\]
 Let us define $\Om_R := \Om \cap B_R$ where $B_R$ is the ball of radius $R$ such that $D \subset B_R$ and let $S_R : H^{1/2}(\partial B_R) \mapsto  H^{-1/2}(\partial B_R) $ be the Dirichlet--to--Neumann map defined for $g \in H^{1/2}(\partial B_R)$ by $S_Rg :={ \partial{u^e}/\partial{r}}|_{{\partial B_R}}$ where $u^e$ is the radiating solution of the Helmholtz equation outside $B_R$  and $u^e = g$ on $\partial B_R$.  Then solving (\ref{PbInit}) is equivalent to find  $u$ in $V_R:=\{v \, \in \, H^1(\Om _R);\, v_{|\pD} \, \in \,  H^1(\pD)\}$ such that:
\begin{equation}
\label{PbBorne}
 \mathcal{P}(\ld,\mu,\pD) \qquad
\cases{
\La u + k^2u = 0  \quad \mbox{inside} \; \Om _R,\\
\Div _{\pD} (\mu \nabla_{\pD} u) +\frac{\partial u}{\partial \nu} + \ld u =0  \quad \mbox{on} \; \pD,\\
 \D{u}{r} - S_R(u) =\D{u^i}{r}-S_R(u^i) \; \mbox{ on } \partial B_R .
}
\end{equation}
Remark that the space $V_R$ equipped with the graph norm is a Hilbert space. We define the operator $A$ of $V_R$ and the bilinear form $a$ of $V_R\times V_R$ by
\[
\fl
 (Au,v)_{V_R} = a(u,v):= \int_{\Om _R} (\nabla u \cdot \nabla \overline{v} -k^2 u\overline{v}) dx 
	+\int_{\pD}(\mu \nabla_{\pD}u \cdot \nabla_{\pD} \overline{v} - \ld u\overline{v})ds
		-\langle S_R u,v \rangle,
\]
for $(u,v) \in V_R\times V_R$ where $\langle \cdot,\cdot\rangle$ is the duality product between $H^{-1/2}(\partial B_R)$ and $H^{1/2}(\partial B_R)$. Furthermore, we define $l$ a linear form on $V_R$ and  $F\in V_R$ by
\[
(F,v)_{V_R}= l(v) :=\int_{\partial B_R} \left(\D{u^i}{r} - S_R(u^i)\right)\overline{v} ds
\]
for all $v\in V_R$. Therefore $u$ is solution to  $\mathcal{P}(\ld,\mu,\pD)$ if and only if 
\begin{equation}
 \label{EqVarForm}
a(u,v) = l(v) \qquad \forall v \in V_R
\end{equation}
or $Au = F$.

\begin{hypothesis}
$(\ld,\mu) \in (L^\infty(\pD))^2$ are such that
\[
 \Im m(\ld)\geq0\,,\; \Im m(\mu)\leq 0 \quad \mbox{ a.e. in } \pD
\]
and there exists $c>0$ such that 
\[
 \Re e(\mu) \geq c \qquad \mbox{ a.e. in } \pD.\]
\end{hypothesis}
In the assumption $\mathcal{H}$, the signs of $\Im m(\ld)$ and $\Im m(\mu)$ are governed by physics, since these quantities represent absorption.
On the contrary, the assumption on $\Re e (\mu)$ is technical and ensures coercivity. However, it is satisfied in the example of a medium with a thin coating which is presented in the introduction.
 In the following,  $K$ will denote a compact set of $(L^\infty(\pD))^2$ such that there exists a constant $c_K>0$ for which assumption $\mathcal{H}$ holds with $c=c_K$ for all $(\ld,\mu) \in K$. 
\begin{proposition}
\label{PropBound}
If assumption $\mathcal{H}$ is satisfied then problem $\mathcal{P}(\ld,\mu,\pD)$ has a unique solution $u$ in $V_R$. In addition there exists a constant $C_K>0$ such that
\begin{equation}
\label{EqUnifBound}
   |||A^{-1}|||\leq C_K 
\end{equation}
for all $(\ld,\mu) \in K$ where $|||\cdot|||$ stands for the operators norm.
\end{proposition}
\begin{proof}
The proof is quite classical and we refer to \cite{BoChHa10} for details.
\end{proof}
Under the sufficient conditions $\mathcal{H}$ on the impedance coefficients $\ld$ and $\mu$ that ensure existence and uniqueness for the forward problem we can study the inverse coefficients problem and this is the aim of the next section. 
\subsection{Formulation of the inverse problem}
We recall the following asymptotic behaviour for the scattered field (see \cite{ColKre}):
\[
 u^s(x) \sim \frac{e^{ikr}}{r^{(d-1)/2}} \left(u^{\infty}(\hat{x}) + O\left(\frac{1}{r}\right)\right)\qquad r \longrightarrow +\infty
\]
uniformly for all the directions $\hat{x} = x/r \in S^{d-1}$. The far--field $u^{\infty} \in L^2(S^{d-1})$ is given by:
\begin{equation}
\label{uinf}
\quad u^{\infty}(\hat{x}) = \int_{\Gamma}\left( u^s(y)\frac{\partial \Phi ^{\infty}(y,\hat{x})}{\partial\nu(y)} - \frac{\partial u^s(y)}{\partial\nu}\Phi ^{\infty}(y,\hat{x}) \right) ds(y) \quad \forall \hat{x} \in S^{d-1},
\end{equation}
 where $\Gamma$ is the boundary  of some regular open domain that contains $D$ and $\Phi ^{\infty}$ is the far--field associated with the  Green function of the Helmholtz equation defined in $\bR^2$ by $ \Phi ^{\infty} (y,\hat{x}):= \frac{e^{i\pi/4}}{\sqrt{8\pi k}} e^{-iky\cdot \hat{x}}$ and in $\bR^3$ by $ \Phi ^{\infty} (y,\hat{x}) := \frac{1}{4\pi} e^{-iky\cdot \hat{x}}$.
 \begin{remark}
 \label{ReLimited}
Since $u^\infty$ is an analytical function on $S^{d-1}$ (see \cite{ColKre}), assuming that  the far--field is known everywhere on $S^{d-1}$ is equivalent to assuming that it is known on a non--empty open set of $S^{d-1}$.
\end{remark}
 Let us define the far--field map
\[
 T : \quad (\ld,\mu,\pD) \rightarrow u^{\infty}
\]
where $u^{\infty}$ is the far--field associated with the scattered field $u^s = u - u^i$ and $u$ is the unique solution of problem $\mathcal{P}(\ld,\mu,\pD)$. The inverse coefficients problem is the following: given  an obstacle $D$, an incident direction $\hat{d} \in S^{d-1}$ and  its associated far--field pattern $u^{\infty}$ for all $\hat{x} \in S^{d-1}$, reconstruct the corresponding impedance coefficients $\ld$ and $\mu$. In other words the inverse problem amounts to invert the map $T$ with respect to the coefficients $\ld$ and $\mu$ for a given $\pD$. The first natural question related to this inverse problem is injectivity of $T$ and stability properties of the inverse map. These questions have been addressed in \cite{BouHad} where for instance results on local stability  in compact sets have been reported. Our subsequent analysis on the stability of the reconstruction of $\ld$ and $\mu$ with respect to perturbed obstacles  will depend on  stability for the inverse map of $T$. We shall first give an improvement of the stability results in \cite{BouHad} for the reconstruction of piecewise constant impedance values. Let us notice that uniqueness (and therefore stability results) with single incident wave fails in general except if one assumes that parts of $\ld$ and $\mu$ are known \textit{a priori}. Moreover we may need to add some restriction for the incident direction or for the geometry of the obstacle (see \cite{BouHad} for more details). 
\subsection{A global stability estimate for the generalized impedance functions}
\label{SecStab}
In this section we shall assume that $\pD$ is a $C^{3}$ boundary and ${\mK}$ is a  compact set of $(L^{\infty}(\pD))^2$ of piecewise constant functions defined as follows: let $I$ be an integer and $(\pD_i)_{i=1,\cdots,I}$ be $I$ non--overlapping open sets of $\pD$ such that $\cup_{i=1}^I\overline{\pD_i} = \pD$. Then  $(\ld,\mu) \in \mK$ if there exists $I$ constants $(\ld_i)_{i=1,\cdots,I}$ and $(\mu_i)_{i=1,\cdots,I}$  respectively such that for $x \in \pD$
\[
   \ld(x) = \sum_{i=1}^I\ld_i \chi_{\,\pD_i}(x) \qquad \mbox{and} \qquad \mu(x) = \sum_{i=1}^I\mu_i \chi_{\,\pD_i}(x).
\]
and there exists $c_{\mK}>$ and $C_{\mK}>0$ such that:
\begin{equation*}
%\label{EqBoundMu}
   c_{\mK} \leq \Re e(\mu_i)\leq C_{\mK}, \quad  -C_{\mK} \leq \Im m(\mu_i) \leq 0
\end{equation*}
and
\[
   0 \leq \Im m(\ld_i) \leq C_{\mK}, \quad |\Re e(\ld_i)|\leq C_{\mK}
\]
for all $i$. From now on, $C_{\mK}$ and $c_{\mK}$ will be generic constants that can change, but they remain independent of $\ld$ and $\mu$. Using that $\mu$ is a piecewise constant function, we shall first explicit a regularity result for the solution  $u$ of the scattering problem.\\ 
For convenience, for sufficiently small $\rho>0$ we denote by $\Xi_R^\rho$ the subset of $\Omega_R$ defined as follows.
If $\delta D$ denotes the set of all points of $\pD$ which are shared by two sets $\overline{\pD_i}$ and $\overline{\pD_j}$ for $i,j=1,..,I$ and $i \neq j$,
we have
\[\Xi_R^\rho =\{x \in \Omega_R,\, d(x,\delta D)>\rho\},\]
where $d$ denotes the distance function.
\begin{lemma}
\label{PropReg}
There exists a constant $C_{\mK}$ (depending on $\rho$ and $R$) such that for all $(\ld,\mu)\in \mK$ the solution $u$ to $\mathcal{P}(\ld,\mu,\pD)$ satisfies
\[
   \|u\|_{H^3(\Xi_R^\rho)} \leq C_{\mK}.
\]
\end{lemma}
\begin{proof}
From the definition of $\mu$ we get
\[
   \mu \Delta_{\pD} u + \D{u}{\nu} +\ld u = 0 \quad \mbox{ on   } \quad \partial \Xi_R^\rho \cap \pD
\]
where $\Delta_{\partial D} u:={\rm div}_{\partial D} (\nabla_{\partial D} u)$ is the Laplace--Beltrami operator. Since $\Delta u = -k^{2}u$ in $\Om_{R}$ we have 
by a standard trace result in space $\{u \in H^1(\Om_R),\,\Delta u \in L^2(\Om_R)\}$
\begin{equation}
\label{trace} 
\left \|\D{u}{\nu}\right\|_{H^{-1/2}(\pD)} \leq C \|u\|_{V_{R}}.
 \end{equation}
Now we consider the regularity of $u$ solving for $f \in H^{-1/2}(\partial \Xi_R^\rho \cap \pD)$ the equation  
 \begin{equation} 
 \label{beltrami}
  \Delta_{\pD} u +\frac{\ld }{\mu}u = f \quad \mbox{ on   } \quad \partial \Xi_R^\rho \cap \pD.
\end{equation}
Following Section 2.5.6 in \cite{Ned01}, by using a local map $\varphi$ and a local parametrization $\xi_i$ of $\partial D$, $i=1,..,d-1$,
the Laplace--Beltrami operator has the local expression
\[\Delta_{\pD} u=\frac{1}{\sqrt{{\rm det} g}}\left(\sum_{i,j=1}^{d-1}\frac{\partial}{\partial \xi_i}\sqrt{{\rm det}g} g^{ij}\frac{\partial u}{\partial \xi_j}\right),\] 
where the vectors $e_i=\partial \varphi/\partial \xi_i$ form a basis in the tangent plane, the matrix $g_{ij}=(e_i.e_j)$ forms the metric tensor $g$, its inverse matrix $g^{-1}$ being denoted $g^{ij}$.
The local regularity of the solution to equation (\ref{beltrami}) on $\partial D$ hence amounts to a regularity problem for a standard elliptic problem in the divergence form in $\mathbb{R}^{d-1}$.
Hence applying the local regularity results for elliptic operators of chapter 8 in \cite{GilTru} with $a^{ij}=\sqrt{{\rm det}g} g^{ij}$, we obtain that
if the second member of the equation is locally in $L^2$ and $\partial D$ is $C^2$ (that is the coefficients $a^{ij}$ are $C^1$), then $u$ is locally in $H^2$.
By using a cut-off function and the interpolation theorem (see \cite{LioMag}) for the equation (\ref{beltrami}), if the second member is locally in $H^{-1/2}$ then $u$ is locally in $H^{3/2}$. Gathering all local estimates on $\partial \Xi_R^\rho \cap \pD$ and using (\ref{trace}) we obtain
 \[
 \| u \|_{H^{3/2}(\partial\Xi_R^\rho\cap\pD)} \leq C_{\mK} \|u\|_{V_{R}}.
 \]
Standard regularity results for the Laplace operator with Dirichlet boundary condition lead to
\[
   \|u\|_{H^2(\Xi_R^\rho)} \leq  C_{\mK} \|u\|_{V_{R}}
\]
and using once again regularity for the Laplace--Beltrami operator we obtain since $\partial D$ is $C^3$,
\[
   \|u\|_{H^{5/2}(\partial\Xi_R^\rho\cap\pD)} \leq C_{\mK}  \|u\|_{V_{R}}.
\]
We finally deduce the desired estimate using regularity result for the Laplace operator and Proposition \ref{PropBound}. \\
\end{proof}
Uniqueness of the reconstruction of $\mu$ (with a known $\ld$) is established in the following Proposition then we will derive a uniform stability estimate for $\mu$.
\begin{proposition}
   \label{PropUniqldmu}
Take $(\ld,\mu^1)$ and $(\ld,\mu^2)$ in ${\mK}$ and assume that their corresponding far--fields  $u^{1,\infty}=T(\ld,\mu^1,\partial D)$ and $u^{2,\infty}=T(\ld,\mu^2,\partial D)$  satisfy
\[
u^{1,\infty}(\hat{x})=u^{2,\infty}(\hat{x})\qquad \forall \hat{x} \in S^{d-1}.
\]
If for all $i=1,...,I$, there exists $\wt{x}_i \in \pD_i$ and $\eta_i>0$ such that $\wt{\pD}_i=\pD_i \cap B(\wt{x}_i,\eta_i)$ is 
\begin{itemize}
\item for $d=2$ either a segment or a portion of a circle,
\item for $d=3$ either a portion of a plane, or a portion of a cylinder, or a portion of a sphere,
\end{itemize}
and the sets $\{x+\gamma \nu(x),\,x \in \wt{\pD}_i,\,\gamma>0\}$ are included in $\Omega$, then $\mu_1 = \mu_2$.
\end{proposition}
\begin{proof}
Since the far--fields $u^{1,\infty}$ and $u^{2,\infty}$ coincide, from Rellich's lemma and unique continuation principle, the associated total fields $u^1$ and $u^2$ coincide up to the boundary $\partial D$, which implies by denoting $u:=u^1=u^2$ that 
\[
\frac{\partial u}{\partial \nu} + {\rm div}_{\partial D} (\mu^j \nabla_{\partial D} u) + \lambda u =0  \quad \mbox{on} \; \partial D
\]
with $j=1,2$. Since  $\ld$ and $\mu^j$ are constant on $\wt{\pD}_i$, $i=1,...,I$ we have
\[
\frac{\partial u}{\partial \nu} + \mu^{j}_i \Delta_{\partial D} u + \lambda_i u =0  \quad \mbox{on} \; \wt{\pD}_i
\]
and
\[
(\mu_{i}^1-\mu_{i}^2) \Delta_{\partial D} u=0 \quad \mbox{on} \;\wt{\pD}_i.
\]
In what follows we focus our attention on some particular portion $\wt{\pD}_i$ and then drop the reference to index $i$ for sake of simplicity.
Suppose that $\mu^1 \neq \mu^2$, 
we have 
\[ 
\Delta_{\partial D}u=0,\quad \D{u}{\nu}= -\ld u \quad \mbox{on} \; \wt{\pD}.
\]
We only consider the case $d=3$ and $\wt{\pD}=\pD \cap B(\wt{x},\eta)$ is a portion of plane of outward normal $\nu$ such that
the set $\wt{Q}=\{x+\gamma \nu,\,x \in \wt{\pD},\,\gamma>0\}$ is included in $\Omega$. We omit the proofs in the other cases because they are very similar (they are addressed in \cite{BouHad} in the case of constant $\ld$ and $\mu$).
There exists a system of coordinates $(x_1,x_2,x_2)$ such that $x(0,0,0)=\wt{x}$  and
\[
\wt{Q}=\{x(x_1,x_2,x_3),\,\sqrt{x_1^2+x_2^2}<\eta,\,x_3 > 0\},
\]
\[
 \wt{\pD}=\{x(x_1,x_2,x_3),\,\sqrt{x_1^2+x_2^2}<\eta,\,x_3 =0\}.
\]
We now consider the function $\wt{u}$ defined in $\wt{Q}$ from $u$ by
\begin{equation*}
    \wt{u}(x_1,x_2,x_3)=u(x_1,x_2,0)c(x_3),
\end{equation*} 
where the function $c$ is uniquely defined by
\begin{equation}
   \frac{d^2 c}{dx^2_3}+ k^2c=0,\quad c(0)=1,\,\frac{dc}{dx_3}(0)=-\ld.
\label{defc}
\end{equation}  
Proceeding as in \cite{BouHad}, we obtain that $u$ and $\wt{u}$ solve the same Helmholtz equation in $\wt{Q}$
and satisfy $\wt{u}=u$ and $\partial_\nu \wt{u}=\partial_\nu u$ on $\wt{\pD}$. Hence, unique continuation implies $\wt{u}=u$ in $\wt{Q}$.\\
Since $u^s$ satisfies the radiation condition 
and 
\[
u(x_1,x_2,x_3)=u^s(x_1,x_2,x_3)+e^{ik(d_1x_1+d_2x_2+d_3x_3)},
\]
we obtain
\begin{equation*}
u(x_1,x_2,x_3) \sim e^{ik(d_1x_1+d_2x_2+d_3x_3)}, \quad x_3 \rightarrow +\infty,
\end{equation*} 
and in particular when $x_1=x_2=0$,
\[
 u(0,0,x_3) \sim e^{ik d_3 x_3}, \quad x_3 \rightarrow +\infty,
\]
that is
\begin{equation}
   c(x_3) \sim C e^{ik d_3 x_3}, \quad x_3 \rightarrow +\infty.
\label{asymp2}
\end{equation} 
To see that the asymptotic behaviour (\ref{asymp2}) is impossible we solve explicitly equation (\ref{defc}) and we get
\[
c(x_3)=\frac{1}{2}\left(1+i\frac{\ld}{k}\right)e^{ik x_3}+\frac{1}{2}\left(1-i\frac{\ld}{k}\right)e^{-ik x_3}.
\]
But from (\ref{asymp2}) we have
\[
   d_3 = 1 \qquad \mbox{ and } \qquad 1-i\frac{\ld}{k} = 0
\]
hence $ \Im m (\ld) = -ik <0$ which is forbidden from assumption $\mathcal{H}$.
\end{proof}
\begin{theorem}
\label{ThStabImpGen}
Under the same assumptions as in Proposition \ref{PropUniqldmu}, there exists a constant $C_{\mK}$ such that for all $(\ld,\mu^1)$ and $(\ld,\mu^2)$ in ${\mK}$ 
\[
   \|\mu^1 - \mu^2\|_{L^\infty(\pD)}\leq C_{\mK} \|T(\ld,\mu^1,\pD) - T(\ld,\mu^2,\pD)\|_{L^2(S^{d-1})}.
\]
\end{theorem}
In order to prove Theorem \ref{ThStabImpGen} we will need the following results on quantification of unique continuation, the proof of which can be found in \cite[Proposition $2.4$]{Bou}  for the first one and in \cite[Lemma $3.1$]{Phu03}  for the second one. 
\begin{proposition}
For all $x_i \in \pD_i \cap \partial \Xi_R^\rho $, there exist $r_i>0$ and an open domain $\omega_i \Subset \Xi_R^\rho$ such that for all $\kappa \in (0,1)$, there exist $\,c,\ve_0>0$ such that for all $\ve \in (0,\ve_0)$, for all $u \in H^2(\Xi_R^\rho)$ which satisfies $\Delta u +k^2u=0$ in $\Xi_R^\rho$, we have
\[
||u||_{H^1(\Xi_R^\rho \cap B(x_i,r_i))} \leq e^{c/\ve}||u||_{H^1(\omega_i)} + \ve^\kappa\, ||u||_{H^2(\Xi_R^\rho)}.
\]
\label{P1}
\end{proposition}
\begin{proposition} 
Let $\omega_0,\omega_1$ be two open domains such that $\omega_0,\omega_1 \Subset \Xi_R^\rho$. There exist
$s,c,\ve_0>0$ such that for all $\ve \in (0,\ve_0)$, for all $u \in H^1(\Xi_R^\rho)$ which satisfies $\Delta u +k^2u=0$ in $\Xi_R^\rho$,
\[||u||_{H^1(\omega_1)} \leq \frac{c}{\ve}||u||_{H^1(\omega_0)}+ \ve^s\, ||u||_{H^1(\Xi_R^\rho)}.\] 
\label{P2}
\end{proposition}
\begin{proof}[Proof of Theorem \ref{ThStabImpGen}]
Let $u^1$ and $u^2$ be the solutions to problems $\mathcal{P}(\ld,\mu^1,\pD)$ and $\mathcal{P}(\ld,\mu^2,\pD)$ respectively with $u^i(x)=e^{ik\hat{d}\cdot x}$ as incident plane wave. Following \cite{Sin07}, we introduce an auxiliary function 
\[
v := \frac{u^2- u^1}{\|\mu^2-\mu^1\|_{L^\infty(\pD)}}
\]
which is solution of the scattering problem
\begin{equation*}
\fl
\cases{
\Delta v + k^2v = 0  \quad \mbox{inside} \; \Om _R,\\
\Div _{\pD} (\mu^1 \nabla_{\pD} v) +\frac{\partial v}{\partial \nu} + \ld v = \frac{1}{\|\mu^2-\mu^1\|_{L^\infty(\pD)}}\Div_{\pD}[(\mu^1 - \mu^2)\nabla_{\pD} u^2]  \quad \mbox{on} \; \pD,\\
 \D{v}{r} - S_R(v) = 0 \; \mbox{ on } \partial B_R.
}
\end{equation*} 
Using  Proposition \ref{PropBound}, there exists a constant $C_{\mK}$ independent of $\ld,\, \mu_1$ and $\mu_2$ such that
\begin{eqnarray*}
\|v\|_{V_R} & \leq \frac{C_{\mK}}{\|\mu^2-\mu^1\|_{L^\infty(\pD)}}\left\| \Div_{\pD}\left[(\mu^1 - \mu^2)\nabla_{\pD} u^2\right] \right\|_{H^{-1}(\pD)} \\
& \leq C_{\mK}\sup_{w\in H^1(\pD)}   \int_{\pD}  \frac{\nabla_{\pD} u^2\cdot \nabla_{\pD}\overline{w }}{\|w\|_{H^1(\pD)}}ds \leq C_{\mK}\|u^2\|_{V_R}
\end{eqnarray*}
and then using Proposition \ref{PropBound} once more for $u^2$ we obtain
\begin{equation*}
\|v\|_{V_R}  \leq C_{\mK}.
\end{equation*}
Using a similar argument as in the proof of Lemma \ref{PropReg} one gets that
\begin{equation}
\label{EqBoundH3}
 \|v\|_{H^3(\Xi^\rho_R)}  \leq C_{\mK}
\end{equation}
for another constant $C_{\mK}$.\\
Let us focus now on the boundary condition for $v$ on $\pD_i$ for a fixed $i=1,..,I$,
\[
\mu_i^1 \Delta_{\pD}v + \D{v}{\nu} + \ld v= \frac{\mu_i^1 - \mu_i^2}{\|\mu^2-\mu^1\|_{L^\infty(\pD)}} \Delta_{\pD} u^2  \,\,\, {\mbox on}\,\,\,\pD_i.
\]  
Using the notation of  Proposition \ref{PropUniqldmu}, for a part $\gamma_i$ of $\wt{\pD_i}$ which is strictly included into $\partial \Xi_R^\rho \cap B(\wt{x_i},r_i)$, where $r_i$ comes from Proposition \ref{P1}, we have
\begin{equation}
\| \mu_i^1 \Delta_{\pD}v + \frac{\partial v}{\partial \nu} + \ld v\|_{L^2(\gamma_i)}=\frac{|\mu_i^2-\mu_i^1|}{\|\mu^2-\mu^1\|_{L^\infty(\pD)}}\|\Delta_{\pD }u^2\|_{L^2(\gamma_i)}.
\label{rel0}
\end{equation}
The strategy now consists in bounding the left-hand side from above and the right-hand side from below. \\ \\
\noindent \textit{(i) Upper bound for the left--hand side of (\ref{rel0})}. Let $\phi_i \in C_0^\infty(B(\wt{x_i},r_i))$ with $\phi_i =0$ on $\overline{\Om_R} \setminus \Xi_R^\rho$ and $\phi_i=1$ on $\gamma_i$ then
\begin{eqnarray*}
\fl
||\mu_i^1 \Delta_{\pD}v + \frac{\partial v}{\partial \nu}  &+ \ld v||_{L^2(\gamma_i)} \leq  C_{\mK} \left[\|\Delta_{\pD} v\|_{L^2(\gamma_i)}+ ||\frac{\partial v}{\partial \nu}||_{L^2(\gamma_i)} + ||v||_{L^2(\gamma_i)} \right]\\
&\leq C_{\mK} \left[\|\Delta_{\pD} (\phi_i v)\|_{L^2(\pD)}+||\frac{\partial (\phi_i v)}{\partial \nu}||_{L^2(\pD)} + ||\phi_i v||_{L^2(\pD)}\right].
\end{eqnarray*}
 By interpolation we get
\begin{eqnarray*}
\fl
 \|\Delta_{\pD} &(\phi_i v)\|_{L^2(\pD)}+||\frac{\partial (\phi_i v)}{\partial \nu}||_{L^2(\pD)} + ||\phi_i v||_{L^2(\pD)}
\\ \fl&\leq C \left(\|\phi_i v\|_{H^2(\pD)}+ \|\frac{\partial(\phi_i v)}{\partial \nu}\|_{L^2(\pD)}\right)\\
\fl&\leq C \left(\|\phi_iv\|_{H^{1/2}(\pD)}^{1/8} \|\phi_iv\|_{H^{5/2}(\pD)}^{7/8} +\|\frac{\partial (\phi_i v)}{\partial \nu}\|_{H^{-1/2}(\pD)}^{1/2}\|\frac{\partial (\phi_i v)}{\partial \nu}\|_{H^{1/2}(\pD)}^{1/2} \right)\\
 \fl&\leq C_{\mK} \|\phi_iv\|_{H^1(\Om_R)}^{1/8}
\end{eqnarray*}
where the last inequality comes from the trace theorems, the fact that $\Delta u = -k^2u$ in $\Om_R$ and the bound (\ref{EqBoundH3}). Consequently
\begin{equation}
\label{EqComp}
||\mu_i^1 \Delta_{\pD}v + \frac{\partial v}{\partial \nu} + \ld v||_{L^2(\gamma_i)} \leq C_{\mK} \|v\|_{ H^1(\Xi^\rho_R\cap B(\wt{x_i},r_i))}^{1/8}.
\end{equation}
By applying Proposition \ref{P1} to $v$, there exists an open set $\om_i \subset \Xi_R^\rho$ such that for some fixed $\kappa \in (0,1)$, there exists $c>0$ such that for small $\ve$,
\[
||v||_{H^1(\Xi^\rho_R \cap B(\wt{x_i},r_i))} \leq e^{c/\ve}||v||_{H^1(\omega_i)} + \ve^\kappa\, ||v||_{H^2(\Xi^\rho_R)}.
\]
Take $R$ sufficiently large such that there exists $\wt{R}$ such that $3\wt{R}+1 < R$ and $ D \subset B_{\wt{R}}$. Then, applying Proposition \ref{P2} to $v$ with $\omega_0  = B(x_0,1/4)$ for $x_0 \in \partial B_{3\wt{R}+1/2} $  and $\om_1 = \om_i$, there exist constants $c,s>0$ such that for small $\ve$,
\[
||v||_{H^1(\omega_i)} \leq \frac{c}{\ve}||v||_{H^1(\omega_0)}+ \ve^s\, ||v||_{H^1(\Xi_R^\rho)}.
\] 
Hence for some fixed $\kappa \in (0,1)$, there exists a constant $c>0$ such that for small $\ve$,
\[
||v||_{H^1(\Xi^\rho_R \cap B(\wt{x_i},r_i))} \leq e^{c/\ve} ||v||_{H^1(\omega_0)}+ \ve^\kappa\, ||v||_{H^2(\Xi_R^\rho)}.
\]
Estimate (\ref{EqBoundH3}) yields to the existence of  constants $C_{\mK},\, c$ such that for small $\ve$,
\begin{equation}
\label{EqProche}
||v||_{H^1(\Xi^\rho_R \cap B(\wt{x_i},r_i))} \leq e^{c/\ve} ||v||_{H^1(\omega_0)}+ C_{\mK}\ve^\kappa.
\end{equation}
From the definition of $x_0$ we get $ B(x_0,3/8) \subset B_{3\wt{R}+1} \setminus \overline{B_{3\wt{R}}}$ and we obtain the following inequalities
\begin{equation}
\label{EqCac}
||v||_{H^1(B(x_0,1/4))} \leq C \| v \|_{L^2(B(x_0,3/8))}\leq C  \| v \|_{L^2( B_{3\wt{R}+1} \setminus \overline{B_{3\wt{R}}})}.
\end{equation}
where the first inequality is obtained by taking $v=\chi^2 \overline{u}$ in the variational formulation (\ref{EqVarForm}) where $\chi \in C^\infty_0(B(x_0,3/8))$ is a positive function and $\chi = 1$ in $B(x_0,1/4)$. In addition, from the continuous embedding of $H^2(\Xi_R^\rho)$ into the space of continuous functions for $d\leq3$ and the uniform bound (\ref{EqBoundH3}) we have the following $L^\infty$ estimate 
\[
\|v\|_{L^\infty(\Xi_R^\rho)} \leq C_{\mK}.
\]
\cite[Lemma $6.1.2$]{Isa98} applies to the radiating solution $v$ of the Helmholtz equation (see also \cite{Bus96}) and there exists $0<\beta_0<1$ such that if 
\[
\beta:=\|v^\infty\|_{L^2(S^{d-1})} \leq \beta_0
\]
then 
\[
 \| v \|_{L^2( B_{3\wt{R}+1} \setminus \overline{B_{3\wt{R}}})} \leq C_{\mK} \beta^{\eta(\beta)}
\]
with
\[
\eta(\beta) := \frac{1}{1 + \mbox{ln}(-\mbox{ln}(\beta)+e)}.
\]
 Plugging the last inequality into (\ref{EqCac}) and then using (\ref{EqProche}) one gets for small $\ve$,
\begin{equation}
\label{EqLointain}
||v||_{H^1(\Xi^\rho_R \cap B(\wt{x_i},r_i))} \leq C_{\mK}( e^{c/\ve} \beta^{\eta(\beta)} + \ve^\kappa)
\end{equation}
and the constants $c$, $C_{\mK}$, $\ve_0$ do not depend on $\ld$ and $\mu$. Now by using the optimization technique of \cite[Corollary $2.1$]{Bou}, which is applicable since $||v||_{H^1(\Xi_R^\rho \cap B(\wt{x_i},r_i))} \leq C_{\mK}$,
it follows from (\ref{EqLointain}) that there exist $C_{\mK},\delta_0>0$ such that for $ \beta^{\eta(\beta)}  \leq\delta_0$,
\begin{equation}
\label{EqLoinProche}
||v||_{H^1(\Xi_R^\rho \cap B(\wt{x_i},r_i))} \leq \frac{C_{\mK}}{\left(\mbox{ln}(C_{\mK}/(\beta^{\eta(\beta)})\right)^\kappa}.
\end{equation}
But $\beta \mapsto \beta^{\eta(\beta)} $ is an increasing function on $(0,1)$ hence there exists $\beta_1\leq \beta_0$ such that inequality (\ref{EqLoinProche}) is satisfied for all $\beta \leq \beta_1$.
By using (\ref{EqComp}) there exist $\kappa,\, C_{\mK},\, \beta_1>0$ such that for $\beta \leq \beta_1$
\begin{equation}
\label{EqLeft}
||\mu_i^1 \Delta_{\pD}v + \frac{\partial v}{\partial \nu} + \ld v||_{L^2(\gamma_i)} \leq\frac{C_{\mK}}{\left(\mbox{ln}(C_{\mK}/(\beta^{\eta(\beta)})\right)^{\kappa/8}}:=f(\beta).
\end{equation}
\textit{(ii) Lower bound for the right--hand side of  (\ref{rel0})}. Let us denote $c^i_{\mK}=\inf_{(\ld,\mu) \in {\mK}} ||\Delta_{\pD} u||_{L^2(\gamma_i)}$ where $u$ is solution to $\mathcal{P}(\ld,\mu,\pD)$. Assume that $c^i_{\mK}=0$. The mapping $(\ld,\mu) \in {\mK} \mapsto ||\Delta_{\pD} u||_{L^2(\gamma_i)}$ is continuous from Proposition \ref{PropBound} and ${\mK}$ is a compact subset of $L^\infty(\pD)^2$. Then $c^i_{\mK}$ is reached for some $(\ld^0,\mu^0) \in {\mK}$.
The corresponding solution $u^0$ satisfies $\Delta_{\pD} u^0=0$ on $\gamma_i$ which is impossible using the same argument as in the uniqueness proof (Proposition \ref{PropUniqldmu}). Therefore $c^i_{\mK}>0$ and from (\ref{rel0}) and (\ref{EqLeft}) we obtain for $c_{\mK} := \min_{i=1,\cdots,I} c^i_{\mK}>0$
\[
\frac{|\mu_i^1-\mu_i^2|}{\|\mu^1-\mu^2\|_{L^\infty(\pD)}} c_{\mK} \leq f(\beta).
\]
for $\beta \leq \beta_1$.
By taking the max over $i$ one obtains
\[
c_{\mK} \leq f(\beta)
\]
and since $f$ is an increasing function, there exists $\beta_2>0$ which is independent of $\ld$ and $\mu$ such that if $\beta\leq\beta_1$ then $\beta\geq \beta_2$. Finally, we have
\[
\min(\beta_1,\beta_2)\leq \beta = \frac{||u^{1,\infty}-u^{2,\infty}||_{L^2(S^{d-1})}}{\|\mu^1-\mu^2\|_{L^\infty(\pD)}},
\]
which is the desired result since $\beta_1$ and $\beta_2$ are independent of $\ld$ and $\mu$.
\end{proof}
Conversely, if one assumes that $\mu$ is known, a slight adaptation of this last proof gives the following stability estimate for $\ld$.
\begin{theorem}
\label{ThStabImpClas}
 There exists a constant $C_{\mK}$ such that for all $(\ld^1,\mu)$ and $(\ld^2,\mu)$ in ${\mK}$ 
\[
   \|\ld^1 - \ld^2\|_{L^\infty(\pD)}\leq C_{\mK} \|T(\ld^1,\mu,\pD) - T(\ld^2,\mu,\pD)\|_{L^2(S^{d-1})}.
\]
\end{theorem}
\begin{proof}
 The proof of this result follows the same lines as the proof of Theorem \ref{ThStabImpGen}. The details are left to the readers.
\end{proof}
We expect that when $I \rightarrow +\infty$, the constant $C_{\mK}$ in Theorems \ref{ThStabImpGen} and \ref{ThStabImpClas} grows exponentially with $I$, as proved for similar Lipschitz stability estimates in \cite{Sin07}. The proof would be based on results established in \cite{Man01,DiCRon03}.
 Let us notice that up to our knowledge, the problem of global stability when both coefficients $\ld$ and $\mu$ are unknowns is still open. For instance, the technique used in the proof of Theorem \ref{ThStabImpGen} cannot be applied in this case.
 \section{A stability estimate in the case of inexact knowledge of the obstacle}
\subsection{The forward and inverse problems for a perturbed obstacle}
\noindent  Here, we are interested in the stability of the reconstruction of $\lambda$ and $\mu$ when exact knowledge of the geometry $\partial D$ is not available. We assume that $D$ is of class $C^1$ and we consider problem (\ref{PbBorne}) with a ``perturbed'' geometry $D_{\ve}$ of $D$ such that one can find a function $\ve \in (C^{1}(\bR^d))^d$ which is compactly supported in $B_R$ and
\[
   \pDe = \{x+\ve(x) \; ; \;  x \in \pD \} \,.
\]
Whenever  
$\| \ve \| <1$, where $\| \cdot \|$ stands for  the $(W^{1,\infty}(\bR^d))^d$ norm
\[
 \|\cdot \| = \|\cdot\|_{(L^{\infty}(\bR^d))^d} + \|\nabla\cdot\|_{(L^\infty(\bR^d))^{d\times d}}, 
\]
the map $f_{\ve} := Id + \ve$ where $Id$ stands for the identity of $\bR^d$, is a $C^1$--diffeomorphism of $\bR^d$ (see chap 5. of \cite{HenPie}). 
\paragraph*{The inverse problem.}
Assume that $\pDe$ is known and that an approximation $\ld_{\ve}$ and $\mu_{\ve}$ of the exact impedance coefficients is available i.e. for some $\delta \geq 0$
\[
   \Vert T(\ld_{\ve},\mu_{\ve},\pDe) - T(\ld,\mu,\pD) \Vert_{L^2(S^{d-1})} \leq \delta \,.
\]
From this information and from the distance between $\pD$ and $\pDe$, can we have an estimate on the boundary coefficients? In other words, do we have 
\begin{equation}
\label{Stab}
 \|\ld_{\ve}\circ f_{\ve} - \ld \|_{L^{\infty}(\pD)}+\|\mu_{\ve}\circ f_{\ve} - \mu \|_{L^{\infty}(\pD)} \leq g(\delta,\ve) 
\end{equation}
for some function $g$ such that $g(\delta,\ve) \rightarrow 0$ as $\delta \rightarrow 0$ and $\ve \rightarrow 0$? In order to prove such a result, we first need a continuity property of $T$ with respect to $\pD$.

\subsection{Continuity of the far--field  with respect to the obstacle}
In the following, we assume that $\| \ve \| <1$.
To prove the continuity of the far--field with respect to the obstacle, we first establish this result for the scattered field. Define $\wt{\ld}_\ve :=\ld\circ f_{\ve}^{-1}$ and $\wt{\mu}_\ve := \mu\circ f_{\ve}^{-1}$ two elements of $L^{\infty}(\pDe)$. To evaluate the distance between the solution $u$ of $\mathcal{P}(\ld,\mu,\pD)$ and the solution $u_{\ve}$ of  $\mathcal{P}(\wt{\ld}_\ve,\wt{\mu}_\ve,\pD_{\ve})$ we first transport $u_{\ve}$ on the fixed domain $\Om_R$ by using the $C^1$--diffeomorphism $f_\ve$ between $\Om_R^{\ve} := \Om_R \setminus \overline{D_\ve}$ and $\Om_R$. 
\begin{figure}
\centering
 \includegraphics[scale=0.3]{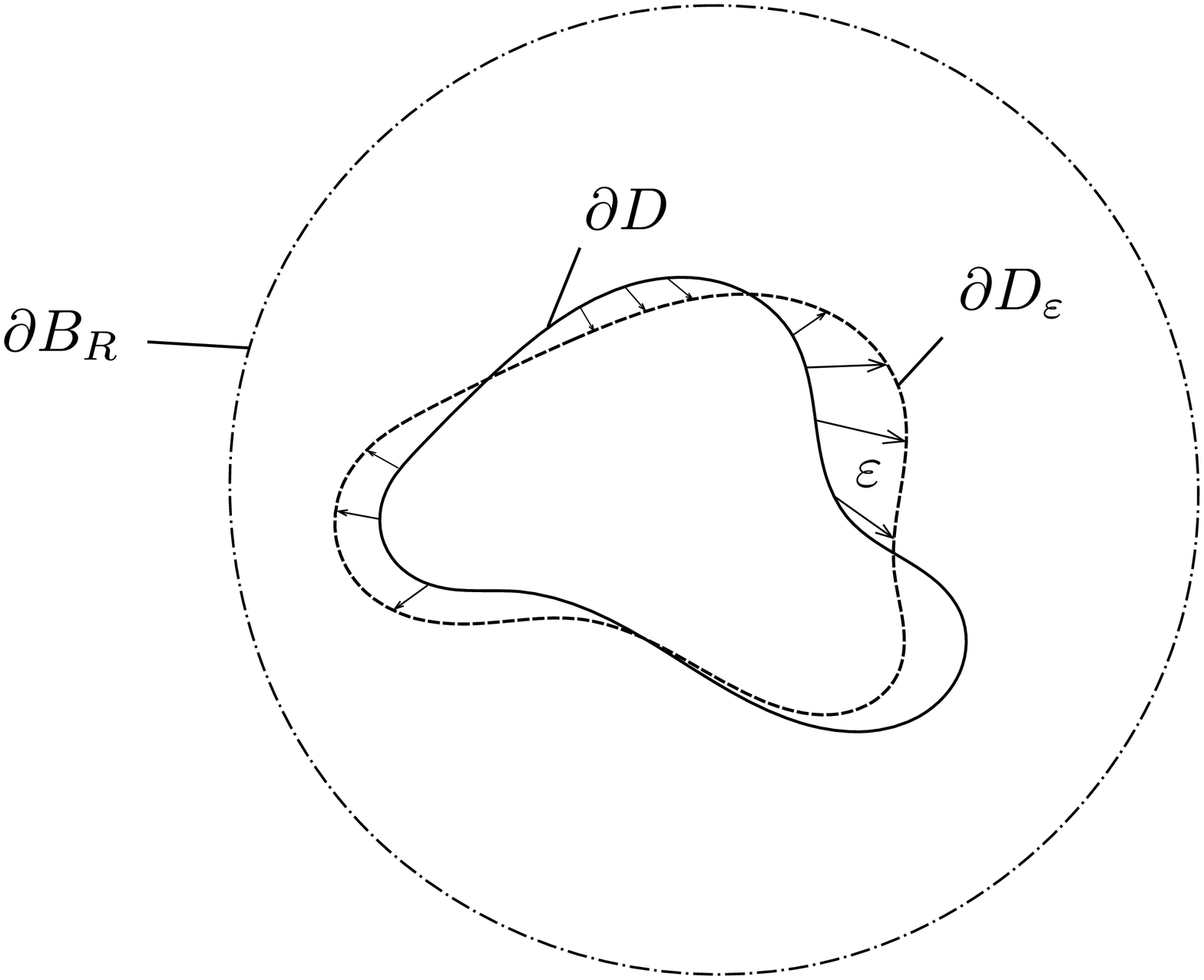}
\caption{Exact and perturbed obstacles.}
\end{figure}
\noindent Define $\wt{u}_{\ve} := u_{\ve} \circ f_{\ve}$ where $u_{\ve}$ is the solution of $\mathcal{P}(\wt{\ld}_\ve,\wt{\mu}_\ve,\pD_{\ve})$, then the following estimate holds.
 \begin{theorem}
\label{ThContScat}
  Let $(\ld,\mu)$ be in $K $. There exists two constants $\ve_0>0$ and $C_K$ which depend only on $K$ such that for all $\|\ve\|\leq \ve_0$ we have:
\begin{equation*}
\label{uinfini}
   \Vert \wt{u}_{\ve}- u  \Vert_{V_R} \leq  C_K \Vert \ve \Vert.
\end{equation*}
\end{theorem}
\noindent We denote by $O(y)$ a $C^{\infty}(\bR)$ function such that  
\begin{equation*}
\label{GrandO}
 |O(y)| \leq C|y| \qquad \forall y \in \bR
\end{equation*}
for $C>0$ independent of $y$. In the proof of the Theorem we will need the following technical Lemma whose proof is postponed to the end of the Theorem's proof.
\begin{lemma}
\label{LeDevGrad}
There exists a constant $C>0$ independent of $\ve$ such that
\begin{eqnarray*}
 \fl\left|\int_{\pDe}(\mu\circ f_\ve^{-1})\nabla_{\pDe}(u\circ f_\ve^{-1})\right.&\left.\cdot \nabla_{\pDe}(v\circ f_\ve^{-1}) ds_\ve -\int_{\pD}\mu\nabla_{\pD} u\cdot \nabla_{\pD} v ds\right|
\\\fl&\leq C\|u\|_{H^1(\pD)}\|v\|_{H^1(\pD)}\|\mu\|_{L^\infty(\pD)}\|\ve\|,
\end{eqnarray*}
 for all $u$, $v$  in  $H^1(\pD)$ and $\mu \in L^{\infty}(\pD)$.
\end{lemma}
\begin{proof}[Proof of Theorem \ref{ThContScat}]
We recall the weak formulation of $\mathcal{P}(\ld,\mu,\pD)$: find $u \in V_R$ such that
\[
a(u,v) = l(v) \quad \forall v \in V_R.
\]
Similarly, the weak formulation of $\mathcal{P}(\wt{\ld}_\ve,\wt{\mu}_\ve,\pDe)$ is: find $u_{\ve} \in V_R^{\ve}:=\{v \, \in \, H^1(\Om _R^\ve);\, v_{|\pDe} \, \in \,  H^1(\pDe)\}$ such that 
\[
   a_{\ve}(u_{\ve},v_{\ve})= l(v_{\ve}) \quad \forall v_\ve \in V_R^{\ve}
\]
where
\begin{eqnarray*}
\fl a_{\ve}(u_{\ve},v_{\ve}) := \int_{\Om^{\ve} _R} (\nabla u_{\ve} \cdot \nabla \overline{v}_{\ve} -k^2 u_{\ve}\overline{v}_{\ve}) dx_\ve  +\int_{\pDe} (\wt{\mu}_\ve \nabla_{\pDe} u_{\ve}\cdot \nabla_{\pDe}\overline{v}_{\ve}-\wt{\ld}_\ve u_{\ve}\overline{v}_{\ve})ds_\ve-\langle S_R u_{\ve},v_{\ve} \rangle.
\end{eqnarray*}
We define a new bilinear form on $V_R$
\[
 \wt{a}_\ve(u,v) := a_\ve(u\circ f_\ve^{-1},v\circ f_\ve^{-1}) \quad  \forall u,\, v \in V_R.
\]
Since $f_\ve(\partial B_R) = \partial B_R$, we have $l(v) = l(v\circ f_\ve^{-1})$ and $\wt{u}_\ve$ is solution of
\[
 \wt{a}_\ve(\wt{u}_\ve,v) = l(v) \quad \forall v \in V_R.
\]
In addition, for all $v\in V_R$, $S_Rv = S_R (v\circ f_{\ve}^{-1})$ and as a consequence for all $u,v \in V_R$ we have $\langle S_R u,v \rangle=\langle S_R (u\circ f_\ve^{-1}),v\circ f_{\ve}^{-1}\rangle$.  Using  the change of variables formula for integrals (see chap. 5 of \cite{HenPie}), we have  :
\begin{eqnarray*}
\fl
  \wt{a}_{\ve}(u,v)=\int_{\Om_R} (\nabla u \cdot P_{\ve}  \cdot\nabla \overline{v} -k^2 u \,\overline{v})J_{\ve} dx&& 
		+\int_{\pDe} \wt{\mu}_\ve [\nabla_{\pDe}(u\circ f_\ve^{-1})]\cdot[\nabla_{\pDe}(\overline{v\circ f_\ve^{-1}})]ds_\ve \\
		 &&-\int_{\pD} \ld u\,\overline{v}J^\nu_{\ve} ds -\langle S_R u,v \rangle 
\end{eqnarray*}
where  $J_{\ve} :=\left|\det(\nabla f_{\ve})\right|$,  $P_{\ve}:=(\nabla f_{\ve})^{-1} (\nabla f_{\ve})^{-T}$ and $J^\nu_{\ve} := J_{\ve}\left|(\nabla f_{\ve})^{-T}\nu\right|$ (for a matrix $B$, $B^{-T}$ denotes the transpose of the inverse of $B$). 
Thanks to Neumann series, we have the following development for $(\nabla f_\ve(x))^{-1} = (Id + \nabla\ve(x))^{-1}$ uniformly for    $x \in \pD$
\[
(\nabla f_\ve)^{-1}(x) = \sum_{n=0}^\infty(-1)^n (\nabla \ve(x))^n = (1+O(\|\ve\|))Id.
\]
As a consequence $P_{\ve}(x)$ expands uniformly for $x \in \pD$ as
\[
   P_{\ve}(x) = Id(1 + O(\|\ve\|)) 
\]
 and we also have
\[
 J_{\ve}(x) = 1 +O(\|\ve\|) \quad  \mbox{and} \quad J^\nu_{\ve}(x) = 1 +O(\|\ve\|)
\]
since $\det(Id + \nabla\ve(x)) = 1+\Div(\ve)+O(\|\ve\|^2)$.  Using all these results and Lemma \ref{LeDevGrad} we infer that for $u$ and $v$ two functions of $V_R$
\[
\fl \left |\int_{\Om_R} (\nabla u \cdot P_{\ve}  \cdot\nabla \overline{v}J_{\ve} -k^2 u \,\overline{v}J_{\ve} - \nabla u  \cdot\nabla \overline{v} + k^2 u \,\overline{v}) dx \right| \leq  C\Vert u \Vert_{V_R} \Vert v \Vert_{V_R} \|\ve\|,
\]
\[
 \fl\left|\int_{\pD}\{\ld u\,\overline{v}J^\nu_{\ve} - \ld u \, \overline{v} \}ds\right|\leq C\Vert \ld \Vert_{L^{\infty}(\pD)}\Vert u \Vert_{V_R} \Vert v \Vert_{V_R}\|\ve\|,
\]
\begin{eqnarray*}
\fl\left| \int_{\pDe}(\mu\circ f_\ve^{-1})\nabla_{\pDe}(u\circ f_\ve^{-1})\right.&\left.\cdot \nabla_{\pDe}(v\circ f_\ve^{-1}) ds_\ve -\int_{\pD}\mu\nabla_{\pD} u\cdot \nabla_{\pD} v ds\right|
\\&\leq C\|u\|_{V_R}\|v\|_{V_R)}\|\mu\|_{L^\infty(\pD)}\|\ve\|,
\end{eqnarray*}
where the constant $C$ does not depend on $\ve$. One obtains for the bilinear forms $\wt{a}_{\ve}$ and $a$:
\begin{equation}
\label{EqEstBil}
     |\wt{a}_{\ve}(u,v) - a(u,v)| \leq  C_K  \Vert u \Vert_{V_R} \Vert v \Vert_{V_R}\Vert \ve \Vert 
\end{equation}
where $C_K$ does not depend on $\ld$, $\mu$ and $\ve$. 
Thanks to the Riesz representation theorem we uniquely define $A_{\ve}$ from $V_R$ into itself by
\[
   (A_{\ve} v,w)_{V_R} = \wt{a}_{\ve}(v,w) \quad \forall\; v,w \; \in V_R.
\]
We recall the definition of the operator $A$ of $V_R$
\[
 (Av,w)_{V_R} = a(v,w) \quad  \forall v,w \in V_R
\]
and $F$ in $V_R$ is defined by
\[
 (F,w)_{V_R} = l(w) \quad  \forall w \in V_R.
\]
Thanks to inequality (\ref{EqEstBil}), we have
\begin{equation*}
   ||| A_{\ve} - A ||| \leq C_K \Vert \ve \Vert .
\end{equation*}
To have information on the scattered field we should have information on the inverse of the operators and we will use once again the results on Neumann series of \cite{Kre}. Actually, as soon as $ |||A^{-1}(A_{\ve} - A)||| \leq 1/2 $ which is true when $\ve_0 \leq 1/(2C_K^2)$ (see Proposition \ref{PropBound}), the  inverse operator of $A_{\ve}$ satisfies
\[
   ||| A_{\ve} ^{-1} ||| \leq \frac{||| A^{-1}|||}{1- |||  A^{-1}(A_{\ve}-A)|||} \leq 2 C_K.
\]
From the identity $\wt{u}_{\ve} - u = A_{\ve}^{-1}(A-A_{\ve}) u$ we deduce 
\begin{eqnarray*}
   \Vert \wt{u}_{\ve}-u \Vert_{V_R} &=& \Vert A_{\ve}^{-1}((A-A_{\ve})u) \Vert_{V_R}\\
&\leq& ||| A_{\ve}^{-1}|||  \Vert (A_{\ve} - A)u \Vert_{V_R}\\
&\leq& 2C_K^2\Vert \ve \Vert\|u\|_{V_R} \leq 2C_K^3\Vert \ve \Vert\|F\|_{V_R} \\
\end{eqnarray*}
where we again used Proposition \ref{PropBound} for the last inequality. This provides the desired estimate.
\end{proof}
\begin{proof}[Proof of Lemma \ref{LeDevGrad}\\]
  Let us consider three functions $u \in H^1(\pD)$, $v \in H^1(\pD)$ and $\mu \in L^\infty(\pD)$ and let $x_0 \in \pD$ . There exists a function $\vp$ of class $C^1$ and an open set $U\subset \bR^{d-1}$ such that
\[
 \pD \cap V = \{\vp(\xi)\;;\;\xi \in U\}
\]
where $V$ is a neighbourhood of $x_0$ and $\vp(0) = x_0$ and such that
\[
 e_i := \D{\vp}{\xi_i}(0)  \;, \mbox{ for } i=1,\cdots,d-1
\]
form a basis of the tangential plane to $\pD$ at $x_0$. We can also use this parametrization to describe a surfacic neighbourhood of $x_{0,\ve} := f_\ve(x_0)$, similarly there exists a neighbourhood $V_\ve$ of $x_{0,\ve}$ such that
\[
 \pDe\cap V_\ve = \{\vp_\ve(\xi)\;;\; \xi \in U\}
\]
where $\vp_\ve := f_\ve \circ \vp$. We define the tangential vectors of $\pDe$ at point $x_{0,\ve} = \vp_\ve(0)$ by
\[
 e_{\ve,i} := \D{\vp_\ve}{\xi_i}(0)  \; \mbox{ for } i=1,\cdots,d-1
\]
and thanks to the chain rule
\begin{equation}
 \label{EqChgt}
 e_{\ve,i} = \nabla f_\ve(x_0) e_i.
\end{equation}
Remark that as $\nabla f_\ve(x_0)$ is an invertible matrix ($\|\ve\|<1$), the family $e_{\ve,i}$ is a basis of the tangent plane to $\pDe$ at $x_{\ve,0}$. Finally, we define the covariant basis of the cotangent planes of $\pD$ at point $x_0$ and of $\pDe$ at point $x_{0,\ve}$  by
\[
 e^i\cdot e_j = \delta_j^i \; \mbox{ and } \; e_{\ve}^i\cdot e_{\ve,j} =  \delta_j^i \quad \mbox{for } i,j =1,d-1.
\]
Using this definition and (\ref{EqChgt})  we have the relation
\[
 e_{\ve}^i = ( \nabla f_\ve )^{-T} e^i \quad i=1,\cdots,d-1.
\]
Finally in the covariant basis, the tangential gradient $\nabla_{\pD}$ for $w \in H^1(\pD)$ at point $x_0$ is
\[
 \nabla_{\pD} w (x_0) = \sum_{i=1}^{d-1}\D{\wt{w}}{\xi_i}(0)e^i
\]
where $\wt{w}:= w\circ \vp$. Similarly, at  point $x_{0,\ve}$ we have for $w_\ve \in H^1(\pDe)$
\[
 \nabla_{\pDe} w_\ve (x_{0,\ve}) =\sum_{i=1}^{d-1} \D{\wt{w}_\ve}{\xi_i}(0)e^i_\ve
\]
where $\wt{w}_\ve :=w_\ve \circ \vp_\ve$. As a consequence, for $w\in H^1(\pD)$
\begin{eqnarray*}
  \nabla_{\pDe} (w\circ f_\ve^{-1})(x_{0,\ve}) &= \sum_{i=1}^{d-1} \D{\wt{w}}{\xi_i}(0)e^i_\ve \\
&= \sum_{i=1}^{d-1}\D{\wt{w}}{\xi_i}(0)(\nabla f_\ve (x_0))^{-T}e^i \\
&=(\nabla f_\ve(x_0) )^{-T} \nabla_{\pD} w (x_0)
\end{eqnarray*}
because $w\circ f_\ve^{-1} \circ \vp_\ve = \wt{w}$. By this formula, we just proved that for all $x_\ve = f_\ve(x),\, x \in \pD$ we have
\begin{equation}
\label{EqChgt2}
 \nabla_{\pDe} (w\circ f_\ve^{-1})(x_{\ve})=(\nabla f_\ve(x) )^{-T} \nabla_{\pD} w (x)
\end{equation}
for all $w \in H^1(\pD)$. For $u$, $v$ and $\mu$, change of variables in the boundary integral ($x=f^{-1}_\ve(x_\ve)$) gives
\begin{eqnarray*}
 \int_{\pDe} &(\mu\circ f_\ve^{-1}) \nabla_{\pDe} (u\circ f_\ve^{-1}) \cdot \nabla_{\pDe} (v\circ f_\ve^{-1}) dx_\ve \\&= \int_{\pD} \mu(x) [\nabla_{\pDe} (u\circ f_\ve^{-1})(f_\ve(x)) ]\cdot [\nabla_{\pDe} (v\circ f_\ve^{-1})(f_\ve(x))] J^\nu_\ve dx
\end{eqnarray*}
 and thanks to the relation (\ref{EqChgt2})
\begin{eqnarray*}
 \int_{\pDe}& (\mu\circ f_\ve^{-1}) \nabla_{\pDe} (u\circ f_\ve^{-1}) \cdot \nabla_{\pDe} (v\circ f_\ve)^{-1} dx_\ve \\&= \int_{\pD} \mu(x) [\nabla_{\pD} u (x) ] (\nabla f_\ve(x) )^{-1}(\nabla f_\ve(x) )^{-T} [\nabla_{\pD} v(x)]J^\nu_\ve dx.
\end{eqnarray*}
Finally recalling that
\[
 \fl (\nabla f_\ve(x) )^{-1}(\nabla f_\ve(x) )^{-T} = P_\ve(x)= (1+O(\|\ve\|))Id \quad \mbox{and} \quad J^\nu_\ve = 1+O(\|\ve\|)
\]
we may write
\begin{eqnarray*}
\fl \int_{\pDe} (\mu\circ f_\ve^{-1}) \nabla_{\pDe} (u\circ f_\ve^{-1}) \cdot \nabla_{\pDe} (v\circ f_\ve^{-1}) dx_\ve = (1+O(\|\ve\|))\int_{\pD} \mu \nabla_{\pD} u\cdot \nabla_{\pD} v dx
\end{eqnarray*}
which is the desired result.
\end{proof}
 \begin{corollary}
\label{ThCont}
  There exists two constants $\ve_0>0$ and $C_K$ such that 
\[
    \Vert T(\ld\circ f_\ve^{-1},\mu\circ f_\ve^{-1},\pDe) -  T(\ld,\mu,\pD) \Vert_{L^2(S^{d-1})} \leq  C_K\Vert \ve \Vert,
\]
for all $(\ld,\mu) \in K$ and $\|\ve\|\leq \ve_0$.
\end{corollary}
\begin{proof}Let $u^{\infty}$ be the far--field that corresponds to the obstacle  $D$  and $u^{\infty}_{\ve}$ the one that corresponds to the obstacle $D_{\ve}$. We use the integral representation formula for the far--field on $\partial B_R$ and as $\wt{u}_\ve|_{\partial B_R} = u_\ve|_{\partial B_R}$ we obtain 
\begin{eqnarray*}
   u^{\infty}_{\ve}(\hat{x})
 &=& \int_{\partial B_R}\left( \wt{u}_{\ve}(y)\frac{\partial \Phi ^{\infty}(y,\hat{x})}{\partial\nu(y)} - \frac{\partial \wt{u}_{\ve}(y)}{\partial\nu}\Phi ^{\infty}(y,\hat{x}) \right) ds(y) \,.
\end{eqnarray*}
The exterior DtN map $S_R$ defined in section \ref{SecDirect} is continuous from $H^{1/2}(\partial B_R)$ to $H^{-1/2}(\partial B_R)$ and as a consequence
\[
   \left\Vert \frac{\partial \wt{u}_{\ve}}{\partial \nu} \right\Vert _{H^{-1/2}(\partial B_R)} \leq C \Vert \wt{u}_{\ve} \Vert _{H^{1/2}(\partial B_R)}\, ,
\]
finally
\[
   \Vert  u^{\infty}_{\ve}(\hat{x})- u^{\infty}(\hat{x}) \Vert _{L^2(S^{d-1})} \leq C \Vert \wt{u}_{\ve} -u \Vert _{H^{1/2}(\partial B_R)} \,.
\]
The trace is continuous from $H^1(\Om_R)$ into $H^{1/2}(\partial B_R)$ so combining this last inequality with (\ref{uinfini}) one obtains the continuity result
\[
     \Vert T(\ld\circ f_\ve^{-1},\mu\circ f_\ve^{-1},\pDe) -  T(\ld,\mu,\pD) \Vert_{L^2(S^{d-1})} \leq C_K\Vert \ve \Vert \,.
\]
\end{proof}

\subsection{A stability estimate of type (\ref{Stab})}
In order to prove a stability estimate of type (\ref{Stab}) we first need to formulate a stability result for the case of an exact geometry. Following the stability results derived in Section \ref{SecStab}, we assume there exists a compact $K \subset (L^\infty(\pD))^2$ such that for $(\ld,\mu)\in K$ there exists a constant $C(\ld,\mu,K)$ which depends on $\ld$, $\mu$ and $K$ such that for all $(\wt{\ld},\wt{\mu}) \in K$ we have
\begin{equation}
\label{EqStabAs}
\fl
 \|\ld - \wt{\ld} \|_{L^{\infty}(\pD)} + \|\mu - \wt{\mu} \|_{L^{\infty}(\pD)}
 \leq C(\ld,\mu,K) \|T(\ld,\mu,\pD) - T(\wt{\ld},\wt{\mu},\pD)\|_{L^2(S^{d-1})} \,.
\end{equation}
We also refer to \cite[Section $4$]{BouHad} for examples of such compact $K$.
\begin{theorem}
\label{ThStabObs}
 There exists a constant $\ve_0$ which depends only on $K$ such that for all $(\ld,\mu) \in K$ there exists a constant $C(\ld,\mu,K)$ such that for all  $\|\ve\|\leq\ve_0$ and for all $(\ld_{\ve}\circ f_{\ve},\mu_{\ve}\circ f_{\ve}) \in K$ that satisfy
\[
    \Vert T(\ld_{\ve},\mu_{\ve},\pD_{\ve}) - T(\ld,\mu,\pD) \Vert_{L^2(S^{d-1})} \leq \delta \quad \mbox{for } \delta >0
\]
we have 
\[
   \Vert \ld_{\ve}\circ f_{\ve} - \ld \Vert_{L^{\infty}(\pD)} + \|\mu_{\ve}\circ f_{\ve} - \mu\|_{L^\infty(\pD)} \leq C(\ld,\mu,K)(\delta +  \|\ve\|).
\]
\end{theorem}
\begin{proof}
 The idea of the proof is first to split the uniform continuity with respect to the obstacle and the stability with respect to the coefficients and to secondly use the stability estimate (\ref{EqStabAs}). We have
\begin{eqnarray*}
\fl  
\| T(\ld_{\ve}\circ f_{\ve},\mu_{\ve}\circ f_{\ve},\pD) - T(\ld,\mu,\pD) \|_{L^2(S^{d-1})}\leq \|& T(\ld_{\ve},\mu_{\ve},\pD_{\ve}) - T(\ld_{\ve}\circ f_{\ve},\mu_{\ve}\circ f_{\ve},\pD) \|_{L^2(S^{d-1})} \\&\quad+ \| T(\ld,\mu,\pD) - T(\ld_{\ve},\mu_{\ve},\pD_{\ve}) \|_{L^2(S^{d-1})}
\end{eqnarray*}
but the hypothesis of the Theorem tells us that
\[
  \Vert T(\ld_{\ve},\mu_{\ve},\pD_{\ve}) - T(\ld,\mu,\pD) \Vert_{L^2(S^{d-1})} \leq \delta
\]
and thanks to the continuity property of the far--field with respect to the obstacle (see Corollary \ref{ThCont}) we have
\[
 \| T(\ld_{\ve},\mu_{\ve},\pD_{\ve}) - T(\ld_{\ve}\circ f_{\ve},\mu_{\ve}\circ f_{\ve},\pD) \|_{L^2(S^{d-1})}\leq C_K \|\ve\|
\]
because  $\ld_{\ve}\circ f_{\ve}$ and $\mu_{\ve}\circ f_{\ve}$ belong to the compact set $K$. Finally
\[
  \| T(\ld_{\ve}\circ f_{\ve},\mu_{\ve}\circ f_{\ve},\pD)-T(\ld,\mu,\pD)  \|_{L^2(S^{d-1})} \leq C_K(\delta+\|\ve\|)
\]
and the stability assumption (\ref{EqStabAs}) implies
\begin{eqnarray*}
\fl  \Vert \ld_{\ve}\circ f_{\ve}- \ld \Vert_{L^{\infty}(\pD)} + \|\mu_{\ve}\circ f_{\ve} - \mu\|_{L^\infty(\pD)}
 \leq C(\ld,\mu,K)  \| T(\ld_{\ve}\circ f_{\ve},\mu_{\ve}\circ f_{\ve},\pD) - T(\ld,\mu,\pD) \|_{L^2(S^{d-1})}
\end{eqnarray*}
which concludes the proof.
\end{proof}
The local nature of this estimate only depends on the local stability result for the impedances. In \cite{BouHad} the reader can find examples of compact sets $K$ for which the stability estimate (\ref{EqStabAs}) holds locally. In the case of a classic impedance boundary condition ($\mu = 0$) global stability results of Sincich in \cite{Sin06} or  of Labreuche in \cite{Lab} can be used to obtain a constant $C(\ld,\mu,K)$ independent of $\ld$ and $\mu$. Furthermore, Theorems \ref{ThStabImpGen} and \ref{ThStabImpClas} also provide global stability results in the case where $\mu\neq0$ and $\ld$ are piecewise constant functions.
 \section{A numerical inversion algorithm and experiments}
 %%%%%%%%%%%%%%%%%%%%%%%%%%%%%%%%%%%%%%%%%%%%%%%%%%%%%%%%%%%%%%%%%
%			Numerique
%%%%%%%%%%%%%%%%%%%%%%%%%%%%%%%%%%%%%%%%%%%%%%%%%%%%%%%%%%%%%%%%%

This section is dedicated to the effective reconstruction of impedance functional coefficients $\ld_0$ and $\mu_0$ from the observed far--field $u^\infty_{obs}:= T(\ld_0,\mu_0,\pD) \in L^2(S^{d-1})$ associated to one or several given incident directions and a given obstacle (which is either exactly known or perturbed). In the simplest case of a single incident wave and an exact knowledge of the obstacle we shall minimize the cost function
\begin{equation}
\label{costfunction}
 F(\ld,\mu) = \frac{1}{2} \Vert T(\ld,\mu,\pD) - u^{\infty}_{obs} \Vert^2_{L^2(S^{d-1})}
\end{equation}
 with respect to $\ld$ and $\mu$ using a steepest descent method. To do so, we first compute the Fr\'echet derivative of $F$. 
\begin{theorem}
\label{ThDiff}
The function $F$ is differentiable for $(\ld,\mu) \in (L^{\infty}(\pD))^2$ that satisfy assumption $\mathcal{H}$ and its Fr\'echet derivative is given by
\[
 \fl \forall (h,l) \, \in (L^{\infty}(\pD))^2 \quad  dF(\ld,\mu) \cdot (h,l)  = \Re e\left< G,{\rm div}_{\pD}(l\nabla_{\pD}u)+ hu \right>_{H^{1}(\pD),H^{-1}(\pD)}
\]
where 
\begin{itemize}
\item $u$ is the solution of the problem $\mathcal{P}(\ld,\mu,\pD)$,
 \item $G=G^i+G^s$ is the  solution of $\mathcal{P}(\ld,\mu,\pD)$ with $u^i$ replaced by
\[
G^i(y) := \int_{S^{d-1}}\Phi^{\infty}(y,\hat{x}) \overline{(T(\ld,\mu,\pD) -u^{\infty}_{obs})}d\hat{x} .
\]
\end{itemize}
\end{theorem}
\noindent To derive such theorem, we have to compute the Fr\'echet derivative of the far--field map $T$ and
hence to prove the following lemma.
\begin{lemma}
\label{LeDiff}
 The far--field map $T$ is Fr\'echet differentiable for $(\ld,\mu) \in (L^{\infty}(\pD))^2$ that satisfy assumption $\mathcal{H}$ and its Fr\'echet derivative $dT(\ld,\mu) : (L^\infty(\pD))^2 \rightarrow L^2(S^{d-1})$ maps $(h,l)$ to $v_{h,l}^\infty$ such that
\[
v_{h,l}^\infty(\hat{x}):=\left< p(.,\hat{x}),{\rm div}_{\pD}(l\nabla_{\pD}u)+hu\right>_{H^1(\pD),H^{-1}(\pD)},\quad \forall \hat{x} \in S^{d-1},\]
where $u$ is the solution of problem (\ref{PbInit}) and $p(.,\hat{x})$ is the solution of (\ref{PbInit}) in which $u^i$ is replaced by $\Phi^\infty(.,\hat{x})$.
\end{lemma}
\begin{proof}
Following the proof of Proposition 6 in \cite{BouHad}, we obtain that $T$ is differentiable and that $dT_{\ld,\mu}(h,l) = v^\infty_{h,l}$ where $v^\infty_{h,l}$ is the far--field associated with $v_{h,l}^s$ solution of 
\begin{equation}
\fl
\cases{
\La v_{h,l}^s + k^2v_{h,l}^s = 0  \quad \textrm{in} \; \Om\\
\Div _{\pD} (\mu \nabla_{\pD} v_{h,l}^s) +\frac{\partial v_{h,l}^s}{\partial \nu} + \ld v_{h,l}^s =   -\Div _{\pD} (l \nabla_{\pD} u) -h u \quad \textrm{on} \; \pD\\
 \displaystyle \lim\limits_{R \to \infty} \int_{\partial B_R}|\partial v^s_{h,l}/\partial r -ikv^s_{h,l}|^2ds =0 \;.}
\end{equation}
From (\ref{uinf}), we have for all $\hat{x} \in S^{d-1}$
\begin{equation}
 \label{EqRepInf}
 v_{h,l}^\infty(\hat{x})=\int_{\pD}\left(v_{h,l}^s\D{\Phi^\infty(.,\hat{x})}{\nu}-\D{v_{h,l}^s}{\nu}\Phi^\infty(.,\hat{x})\right) ds.
\end{equation}
Since on $\pD$ 
\[
\D{v^s_{h,l}}{\nu}=-\Div_{\pD}(\mu \nabla_{\pD}v^s_{h,l}) -\ld v^s_{h,l}-\Div _{\pD} (l \nabla_{\pD} u) -h u,
\]
we obtain with integration by parts
\begin{eqnarray}
\label{eqvinf}
\fl
  v_{h,l}^\infty(\hat{x})=&\left<v_{h,l}^s,\Div_{\pD}(\mu \nabla_{\pD}\Phi^\infty(.,\hat{x})) + \D{\Phi^\infty(.,\hat{x})}{\nu}+ \ld\Phi^\infty(.,\hat{x}) \right>_{H^{1}(\pD),H^{-1}(\pD)} \nonumber\\
&+ \left< \Phi^\infty(.,\hat{x}), \Div_{\pD}(l \nabla_{\pD}u) + hu \right>_{H^{1}(\pD),H^{-1}(\pD)} \,.
\end{eqnarray}
We introduce the solution $p(\cdot,\hat{x})$ of (\ref{PbInit}) with $u^i = \Phi^{\infty}(\cdot,\hat{x})$. 
The associated scattered field
$p^s(\cdot,\hat{x}) := p(\cdot,\hat{x}) - \Phi^{\infty}(\cdot,\hat{x})$ satisfies on $\pD$
\[
 \Div _{\pD} (\mu \nabla_{\pD} p^s) +\frac{\partial p^s}{\partial \nu} + \ld p^s = -\Div_{\pD}(\mu \nabla_{\pD}\Phi^\infty) - \D{\Phi^\infty}{\nu}- \ld\Phi^\infty \,.
\]
Since $v^s_{h,l}$ and $p^s$ are radiating solutions of a scattering problem the following identity holds:
\begin{equation*}
 \int_{\pD}\left(\D{p^s}{\nu}v^s_{h,l} - p^s \D{v^s_{h,l}}{\nu}\right)ds = 0.
\end{equation*}
Using the boundary condition for $p^s$ and $v_{h,l}^s$, equation (\ref{eqvinf}) becomes 
\begin{eqnarray*}
  v_{h,l}^\infty(\hat{x})&=&-\left<p^s(\cdot,\hat{x}),\Div_{\pD}(\mu \nabla_{\pD}v^s_{h,l}) + \D{v^s_{h,l}}{\nu}+ \ld v^s_{h,l} \right>_{H^{1}(\pD),H^{-1}(\pD)} \\
& &+ \left< \Phi^\infty(.,\hat{x}), \Div_{\pD}(l \nabla_{\pD}u)+ hu \right>_{H^{1}(\pD),H^{-1}(\pD)} \\
& =& \left<p(.,\hat{x}), \Div_{\pD}(l \nabla_{\pD}u) +hu \right>_{H^{1}(\pD),H^{-1}(\pD)} \,
\end{eqnarray*}
which completes the proof.
\end{proof}
We are now in a position to prove Theorem \ref{ThDiff}.
\begin{proof}[Proof of Theorem \ref{ThDiff}]
By composition of derivatives and using the Fubini theorem we have for $(h,l) \, \in (L^{\infty}(\pD))^2$
\begin{eqnarray*}
  dF(\ld,\mu)\cdot(h,l) &= \Re e\left\{(T(\ld,\mu,\partial D)-u^{\infty}_{obs}, dT(\ld,\mu)\cdot (h,l))_{L^2(S^{d-1})}\right\} \\
&= \Re e\int_{S^{d-1}}\left\{\overline{ (T(\ld,\mu,\partial D)-u^{\infty}_{obs})}(\hat{x})\left< p(y,\hat{x}),A(u)(y)\right>_{H^1(\pD),H^{-1}(\pD)}\right\} d\hat{x} \\
&= \Re e \left< G , A(u)\right>_{H^1(\pD),H^{-1}(\pD)}
\end{eqnarray*}
with 
\[
 A(u)(y) =  \Div_{\pD}(l(y)\nabla_{\pD}u(y)) + h(y)u(y).
\] 
\end{proof}
\subsection{Numerical algorithm}
To minimize the cost function (\ref{costfunction}) we use a steepest descent method and we compute the gradient of $F$ with the help of  Theorem \ref{ThDiff}. We solve the direct problems using a finite element method (implemented with FreeFem++ \cite{FF++}) applied to (\ref{PbBorne}). We look for the imaginary part of a function $\ld$  with $\Im m(\ld) \geq 0$ and the real part of  a function $\mu $ with $\Re e(\mu) (x)\geq c> 0$ for almost every $x \in \pD$ assuming that $\Re e(\ld)$ and $\Im m(\mu)$ are known in order to satisfy hypothesis presented in \cite{BouHad} for which uniqueness and local stability hold. Moreover, for sake of simplicity we will choose these known parts of the impedances equal to zero. Let us give initial values $\ld_{init}$ and $\mu_{init}$ in the same finite element space as the one used to solve the forward problem. We update these values at each time step $n$ as follows
\[
 \ld_{n+1} = \ld_n -i\delta\ld_n
\]
where the  descent direction $\delta\ld_n$ is taken proportional to $dF(\ld_n,\mu_n)$. Since the number of parameters for $\ld_n$ is in general high, a regularization of $dF(\ld_n,\mu_n)$ is needed. We choose to use a $H^1(\pD)$--regularization  (see \cite{All} for a similar regularization procedure) by taking $\delta\ld_n \in H^1(\pD)$ solution to
\begin{equation}
\label{EqGrad}
 \fl \qquad \eta_1 (\nabla_{\pD} (\delta \ld_n), \nabla_{\pD} \vp)_{L^2(\pD)} + ( \delta \ld_n, \vp)_{L^2(\pD)} = \alpha_1 dF(\ld_n,\mu_n)\cdot(i\vp,0) 
\end{equation}
for each $\vp$ in the finite element space and with $\eta_1$ the regularization parameter  and $\alpha_1$  the descent coefficient for $\ld$.  We apply a similar procedure for $\mu$,
\[
\mu_{n+1} = \mu_{n} - \delta \mu_n
 \]
 where $\delta \mu_n$  solves
\[
\fl \qquad \eta_2 (\nabla_{\pD} (\delta \mu_n), \nabla_{\pD} \vp)_{L^2(\pD)} + ( \delta \mu_n, \vp)_{L^2(\pD)} = \alpha_2 dF(\ld_n,\mu_n)\cdot(0,\vp) .
\]
We take two different regularization parameters for $\ld$ and $\mu$ since we observed that the algorithm has different sensitivities with respect  to each coefficient. From the practical point of view we choose large $\eta_i$ at the first steps to quickly approximate the searched $\ld$ and $\mu$ then we decrease them during the algorithm in order to increase the precision of the reconstruction. In all the computations (except for  constant $\ld$ and $\mu$), the parameters $\alpha_1$ and $\alpha_2$ are chosen in such a way that the cost function decreases at each step. Finally, we update alternatively $\ld$ and $\mu$ because the cost function is much more sensitive to $\ld$ than to $\mu$ and as a consequence, if we update both at each time step, we would have a poor reconstruction of $\mu$. Concerning the stopping criterion, we stop the algorithm if the descent coefficients $\alpha_1$ and $\alpha_2$ are too small or if the number of iterations is larger than $100$ (in every cases there was not a significant improvement of the reconstruction after $80$ iterations). 

\subsection{Numerical experiments}
In this section we will show some numerical reconstructions using synthetic data generated with the code \textit{FreeFem++}  in two dimensions to illustrate the behaviour of our numerical method. First of all, we will see that the use of a single incident wave is not satisfactory and we will quickly turn to the use of several incident waves. Then all the simulations will be done with several incident waves and with limited aperture data. Remark that all the theoretical results still hold in this particular case (see remark \ref{ReLimited}). In all the simulations the obstacle is an ellipse of semi-axis $0.4$ and $0.3$, its diameter is hence more or less equal to the wavelength $2\pi/k$ when $k=9$. \\
Moreover, since we consider a modelling of physical properties we rescale the equation on the boundary of the obstacle $\pD$ in order to deal with dimensionless coefficients $\ld$ and $\mu$. The equation on $\pD$ becomes
\[
 \Div_{\pD}\left(\frac{\mu}{k}\nabla_{\pD}u \right) + \D{u}{\nu} + k\ld u=0.
\]
In all cases (except when we specify it), we simply reconstruct $\mu$ taking $\ld=0$ because the reconstruction of $\ld$ has been investigated for a long time (see \cite{CaCoMo} or more recently \cite{CaCoMo10}). Finally, as we consider star--shaped obstacles we can define the impedance functions as functions of the angle $\theta$. In the following we will represent $\ld$ and $\mu$ with the help of such a parametrization. Other experiments have been carried out and can be found in \cite{BoChHa10}.

\subsubsection{A single incident wave with full aperture}
In this section we consider the exact framework of the theory developed at the beginning,
\begin{figure}[h]
\begin{minipage}{\textwidth}
 \includegraphics[width =.49\textwidth]{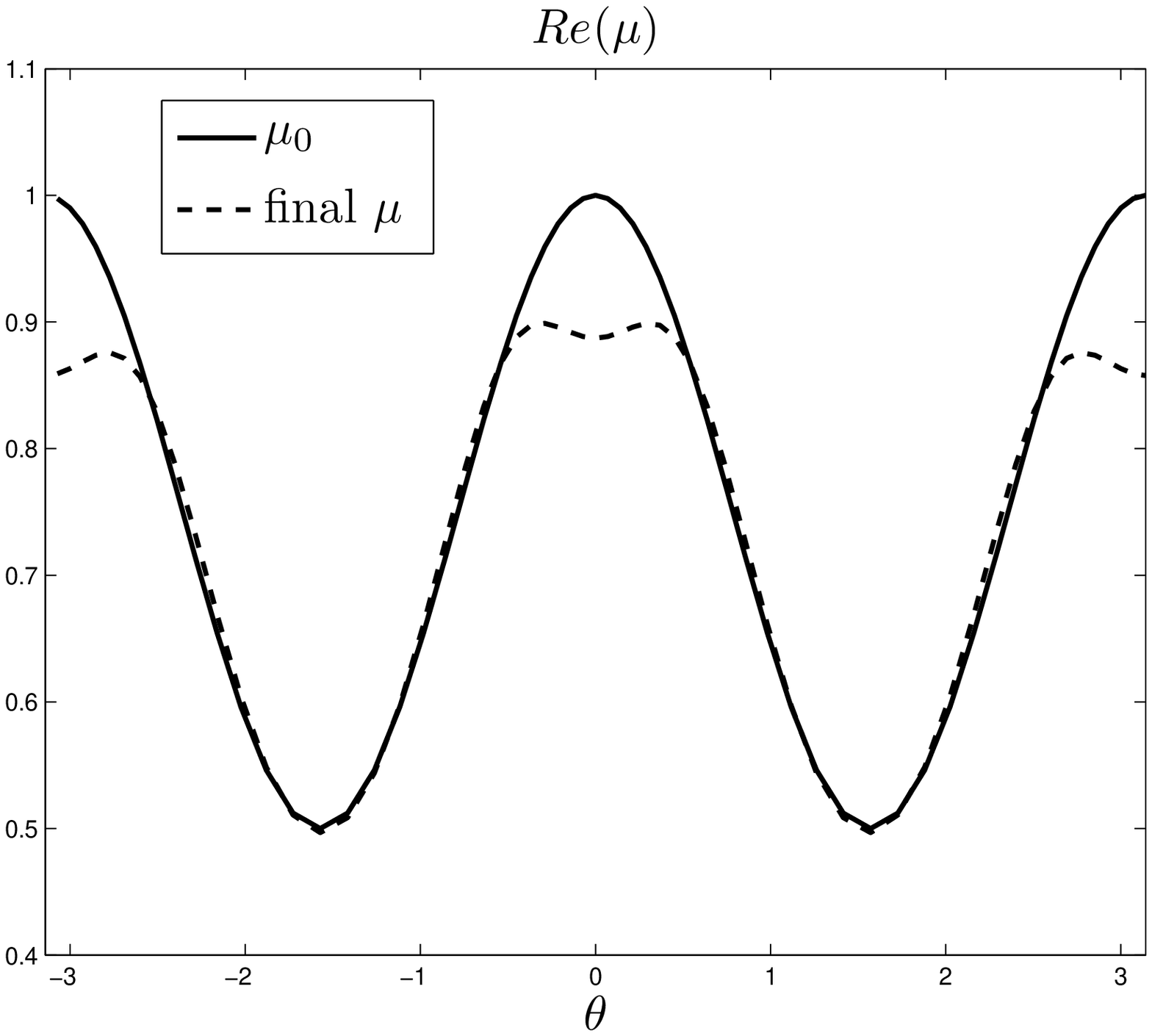}
 \includegraphics[width =.49\textwidth]{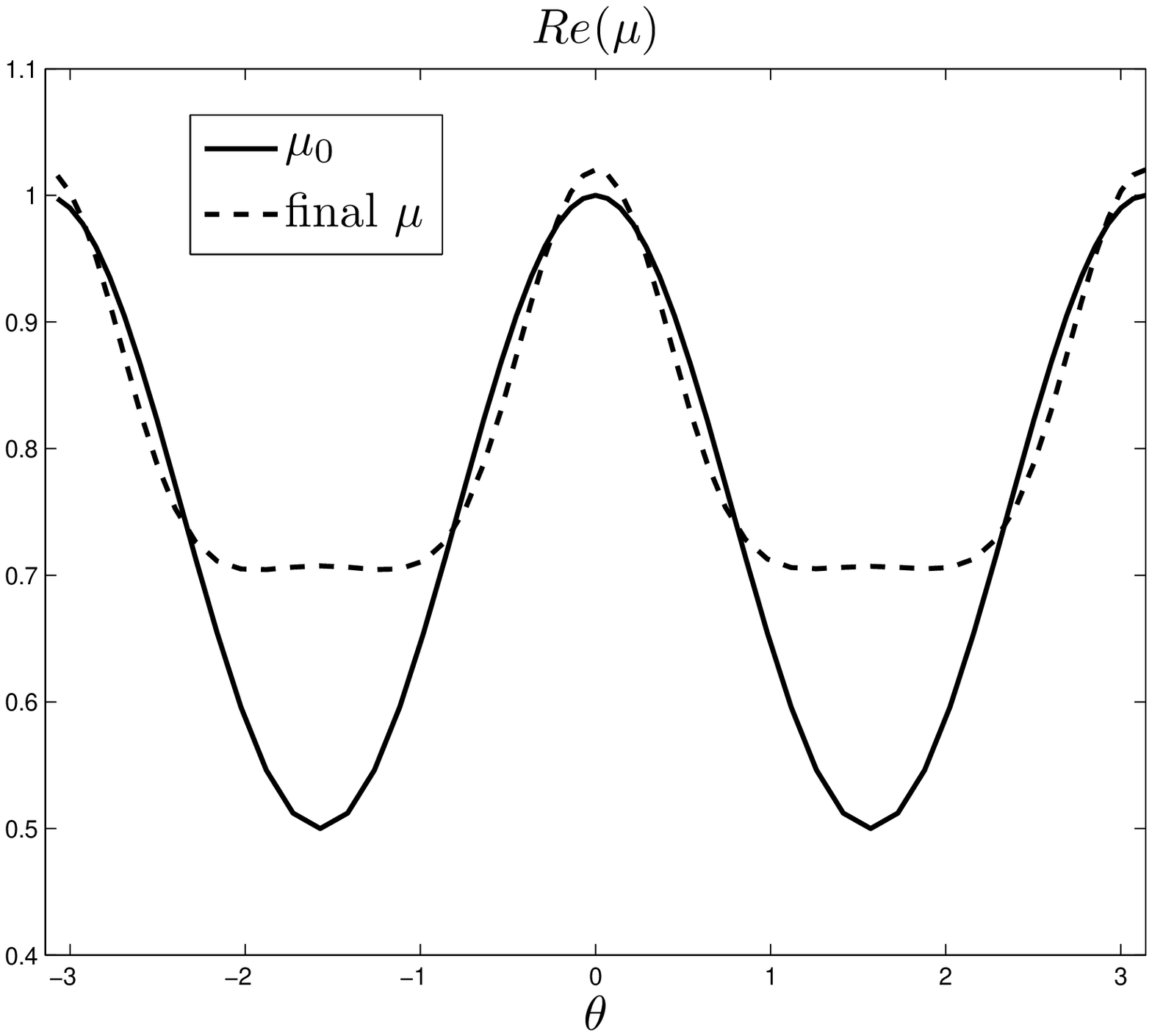}
\caption{Reconstruction of $\Re e(\mu_0)=0.5(1+\cos^2(\theta))$, $\ld = 0$, $\mu_{init} = 0.7$, wave number $k = 9$, the incident angle is $0$ on the left and $\pi/2$ on the right. }
\label{FigUnique}
\end{minipage}
\end{figure}
 namely we enlighten the obstacle with a single incident plane wave and we measure the far--field in all directions. As we expected, the reconstruction is quite good in the enlightened area but rather poor far away from such area (see Figure \ref{FigUnique}) that's why we will consider a several incident waves framework.

\subsubsection{Several incident waves and limited aperture}
From now on we suppose that we measure several far--fields corresponding to several incident directions. We hence reproduce an experimental device that would rotate around the obstacle $D$. To be more specific, we denote $u^s(\cdot,d)$ the scattered field associated with the incident plane wave of direction $d$. Considering we have $N$ incident directions $d_j$ and  $N$ areas of observation $S_j \subset S^1$ such that the angle associated to $S_j$ is $2\pi/N$, we construct a new cost function
\begin{eqnarray*}
   F(\ld,\mu) &=&\frac{1}{2} \sum_{j=1}^N \Vert T(\ld,\mu,\pD,d_j) - u^{\infty}_{obs}(\cdot,d_j)\Vert^2_{L^2(S_j)} \,.
\end{eqnarray*}
To minimize $F$ we use the same technique as before but now the Herglotz incident wave $G^i$ is associated with the incident direction $d_j$ and is given by
\[
 G^i(y,d_j) =  \int_{S_j}\Phi^{\infty}(y,\hat{x}) \overline{(T(\ld,\mu,\pD,d_j) -u^{\infty}_{obs}(\cdot,d_j))}d\hat{x}.
\]
More precisely, for the next experiments we send $N=10$ incident waves uniformly distributed on the unit circle and hence the $S_j$ are portions of the unit circle of aperture $\pi/5$. In order to evaluate the convergence of the algorithm, we introduce the following relative cost function:
\[
\mbox{Error }:=\sqrt{\frac{\sum_{j=1}^N \Vert T(\ld,\mu,d_j) - u^{\infty}_{obs}(\cdot,d_j)\Vert^2_{L^2(S_j)} }{\sum_{j=1}^N \Vert u^{\infty}_{obs}(\cdot,d_j)\Vert^2_{L^2(S_j)}}} \,.
\]
Moreover, we add some noise on the data to avoid "inverse crime". Precisely we handle some noisy data $u^{\infty}_{\sigma}(\cdot,d_j)$ such that
\[
 \frac{\|u^{\infty}_{\sigma}(\cdot,d_j)-u^{\infty}_{obs}(\cdot,d_j)\Vert_{L^2(S_j)}}{  \Vert u^{\infty}_{obs}(\cdot,d_j)\Vert_{L^2(S_j)}}= \sigma.
\]
 In the next experiments we study the impact of the level of noise ($1\%$ and $5\%$) on the quality of the reconstruction. $\mbox{Error}(\sigma)$ will denote the final error with amplitude of noise $\sigma$.

\paragraph*{Influence of the wavelength and the regularization parameter.}
First of all we are interested in the influence of the wavelength on the accuracy of the results, the first two graphics on Figure \ref{FigFreq} show how the algorithm behaves with respect to the wavelength. We can see that if we decrease the wavelength (Figure \ref{FigFreq24}) the reconstructed impedance is very irregular, that's why on Figure \ref{FigFreqReg} we add some regularization to flatten the solution,  and then improve the reconstruction compared to Figure \ref{FigFreq24}.
\begin{figure}[h!]
\begin{minipage}{\textwidth}
\centering
 \subfigure[$ \mbox{Error }(1\%) =2 \%\,,\; \mbox{Error }(5\%) =11 \%$]{\includegraphics[width =.49\textwidth]{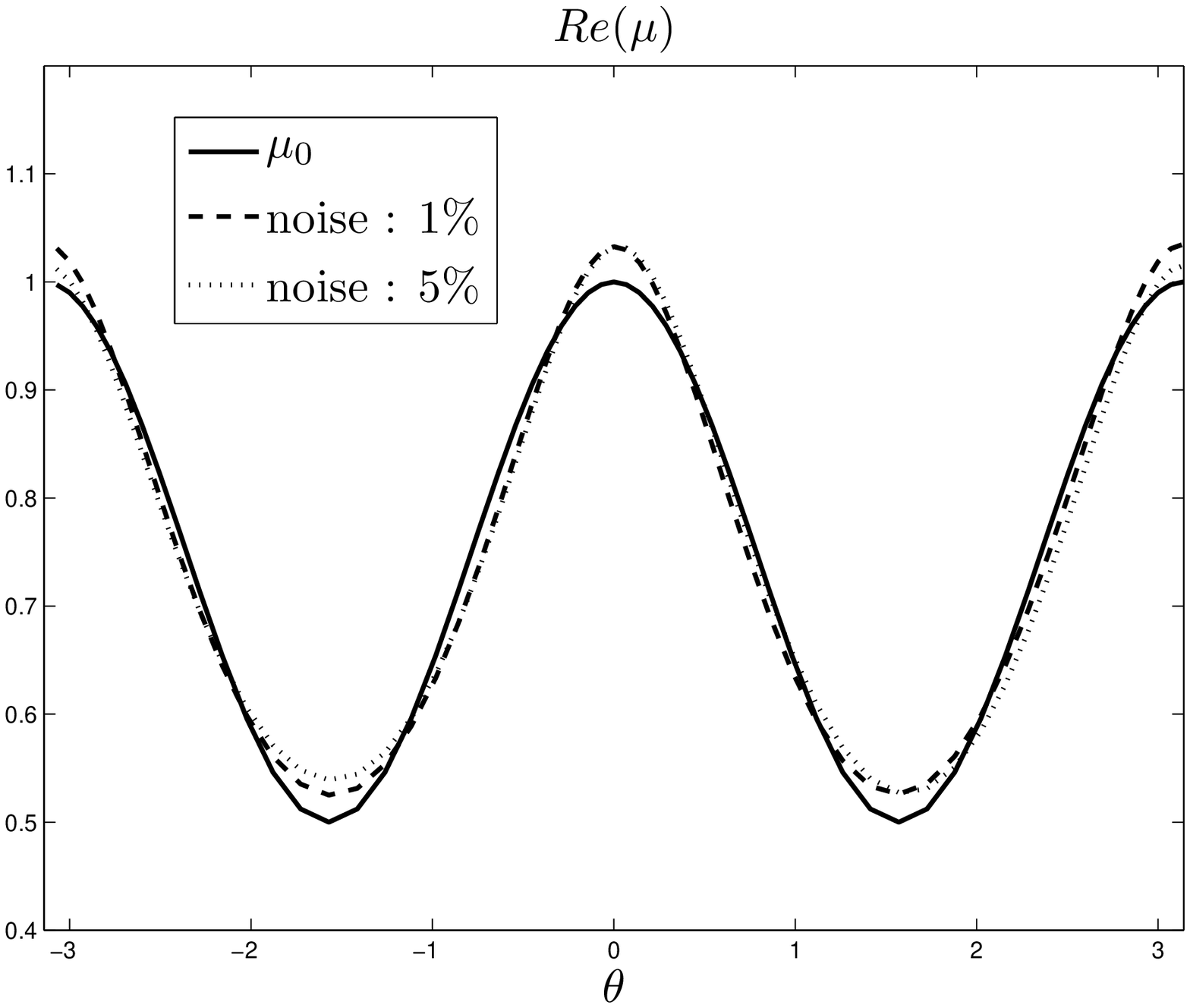}}
 \subfigure[$ \mbox{Error }(1\%) =2.8 \%\,,\; \mbox{Error }(5\%) =11.5 \%$]{\label{FigFreq24} \includegraphics[width =.49\textwidth]{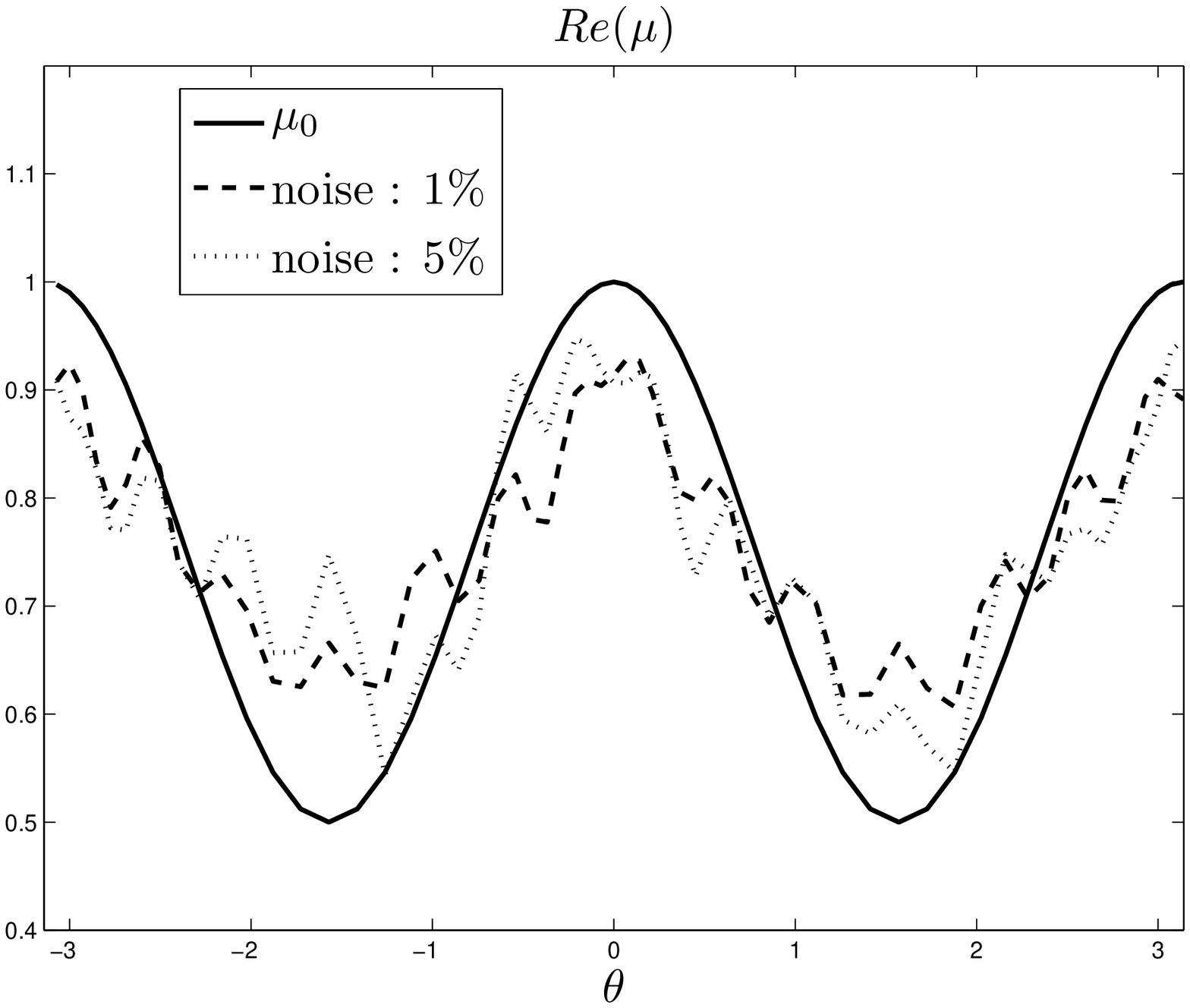}}
\subfigure[$ \mbox{Error }(1\%) =2.1 \%\,,\; \mbox{Error }(5\%) =10.2 \%$]{\label{FigFreqReg} \includegraphics[width =.49\textwidth]{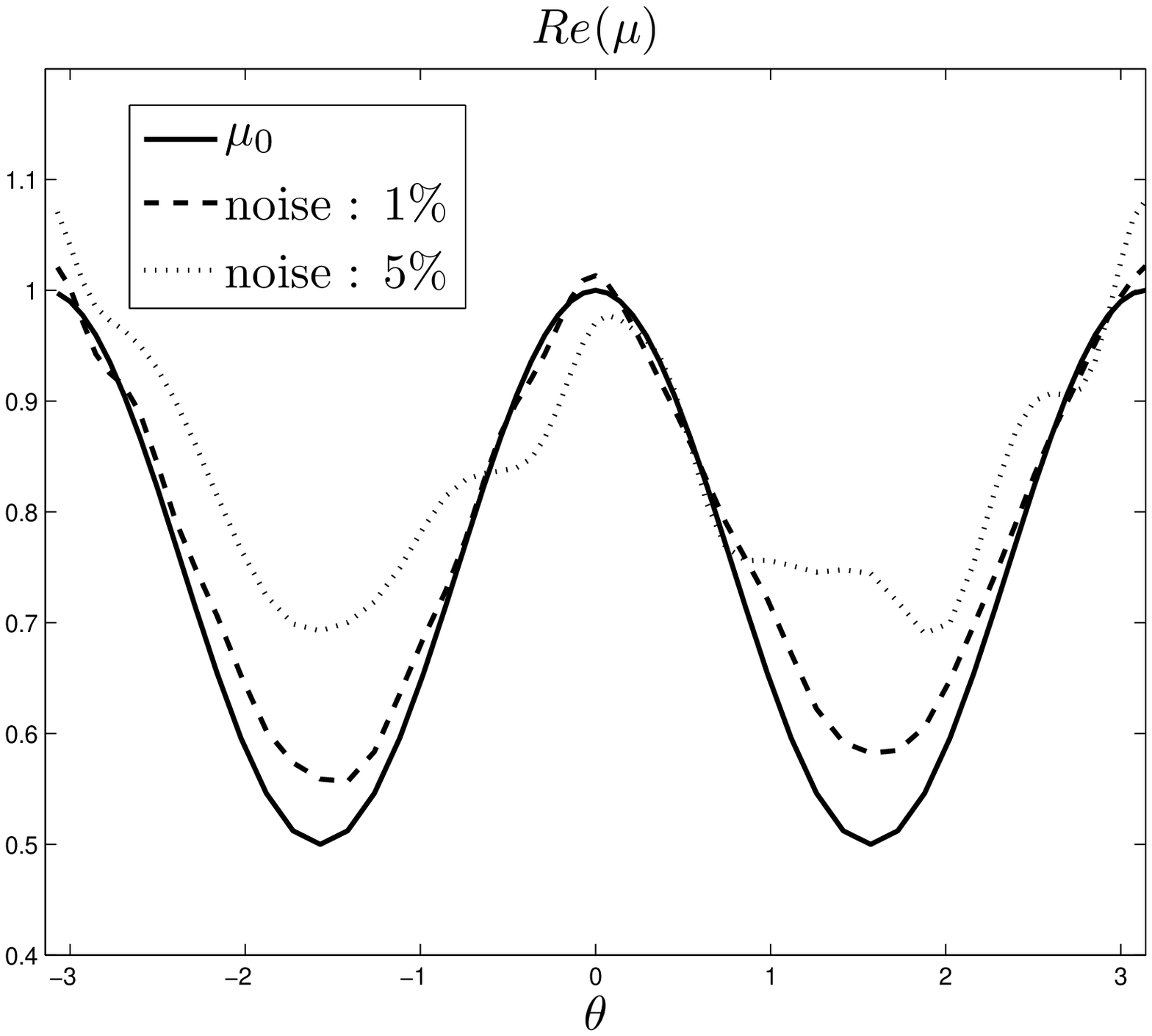}}
\caption{Reconstruction of $\Re e(\mu_0)=0.5(1+\cos^2(\theta))$, $\mu_{init} = 0.7$, $\ld = 0$, with no regularization procedure, wave number $k=2$ (top left) and $k=24$ (top right); with a regularization procedure and $k=24$ (bottom). }
\label{FigFreq}
\end{minipage}
\end{figure}

\paragraph*{The case of non--smooth functional coefficients.}
We are able to handle a non--smooth coefficient $\mu$ since $\mu$ is expressed as a linear combination of functions of the finite element space. 
\begin{figure}[h] 
\begin{minipage}{\textwidth}
  \subfigure[$ \mbox{Error }(1\%) =2.9 \%\,,\; \mbox{Error }(5\%) =11.4 \%$]{\includegraphics[width =.49\textwidth]{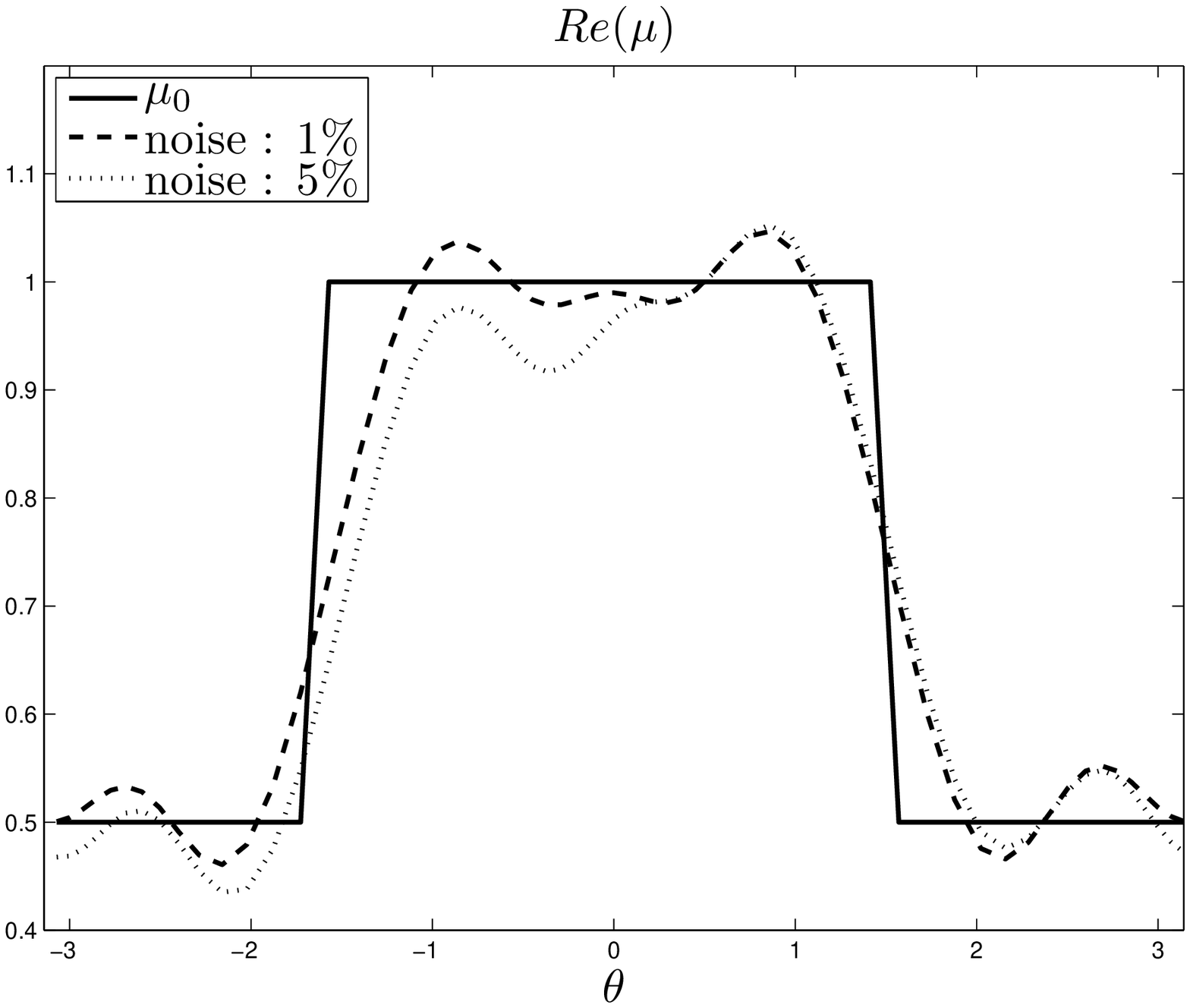} }
  \subfigure[$ \mbox{Error }(1\%) =2.3 \%\,,\; \mbox{Error }(5\%) =11 \%$]{\includegraphics[width =.49\textwidth]{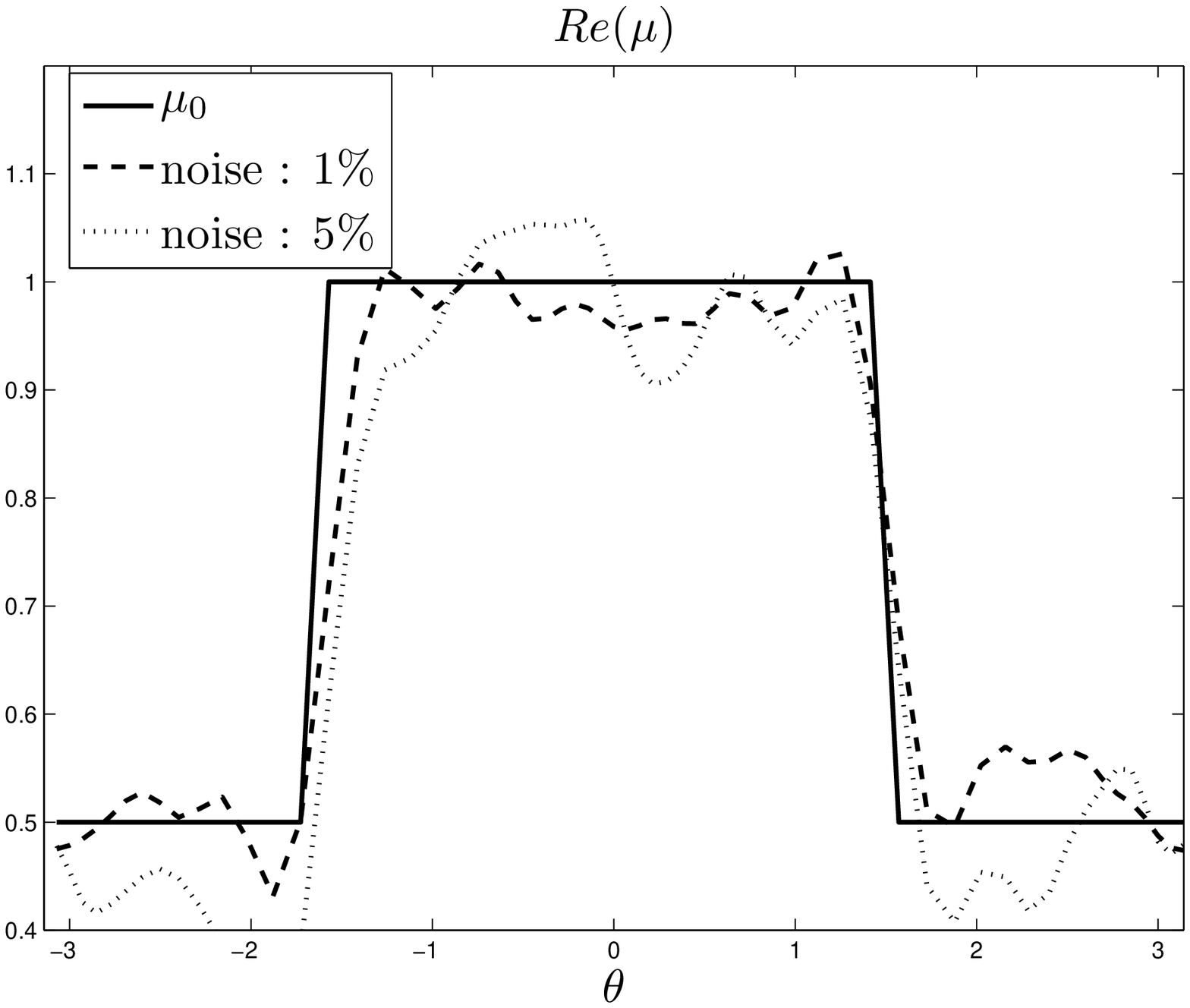}\label{FigConstb}}
\caption{Reconstruction of $\Re e(\mu_0)=0.5 +0.5\chi_{[-\pi/2,\pi/2]}$, $\mu_{init} = 0.7$, $\ld = 0$, the wave number is $k=9$ on the left and $k=24$ on the right.}
\label{FigConst}
\end{minipage}
\end{figure}
 We present our results on Figure \ref{FigConst} for piecewise constant functions $\mu$. To have a good reconstruction of a piecewise constant function, we need a small wavelength. However, we have  just seen before that too small wavelength generates instability due to the noise that contaminates data. That's why we use a two steps procedure. First we use  a large wavelength equal to $0.7$ ($k=9$, on the left) to quickly find a good approximation of the coefficient. Secondly, to improve the result, we use a three times smaller  wavelength  ($k=24$ on the right).  Hence, we combine the advantage of a large wavelength (low numerical cost) and the advantage of a small wavelength (good precision on the reconstruction of the discontinuity).

\paragraph*{Simultaneous search for $\ld$ and $\mu$.}
We now study the simultaneous reconstruction of $\ld$ and $\mu$.
\begin{table}[h]
\centering
\begin{tabular}{|c|c|c|c|c|c|c|}
\hline $\ld_{0}$ & $\mu_{0}$ & $\ld_{init}$ & $\mu_{init}$ & Reconstructed $\ld$ & Reconstructed $\mu$ & Error on the far--field\\ 
\hline  $i$ & $1$  & $0.5i$ &  $0.5$& $i$ & $0.97$ & $0.05\%$ \\ 
\hline  $i$ & $0.2$ & $0.5i$ &  $0.1$& $0.99i$ & $0.21$  & $0.08\%$ \\ 
\hline  $i$ &   $5$ & $0.5i$ &  $2.5$ & $0.97i$ &$5.12$ &$0.07\%$ \\ 
\hline 
\end{tabular}
\caption{Simultaneous reconstruction of constant $\ld$ and $\mu$ with the wave number $k=9$; $\sigma = 1\%$ of noise.}
\label{TabConst}
\end{table} 
Table \ref{TabConst} represents the simultaneous reconstruction of  constant $\ld$ and $\mu$ and we observe a good reconstruction in the constant case. 
\begin{figure}[h]
\begin{minipage}{\textwidth}
 \subfigure{\includegraphics[width =.49\textwidth]{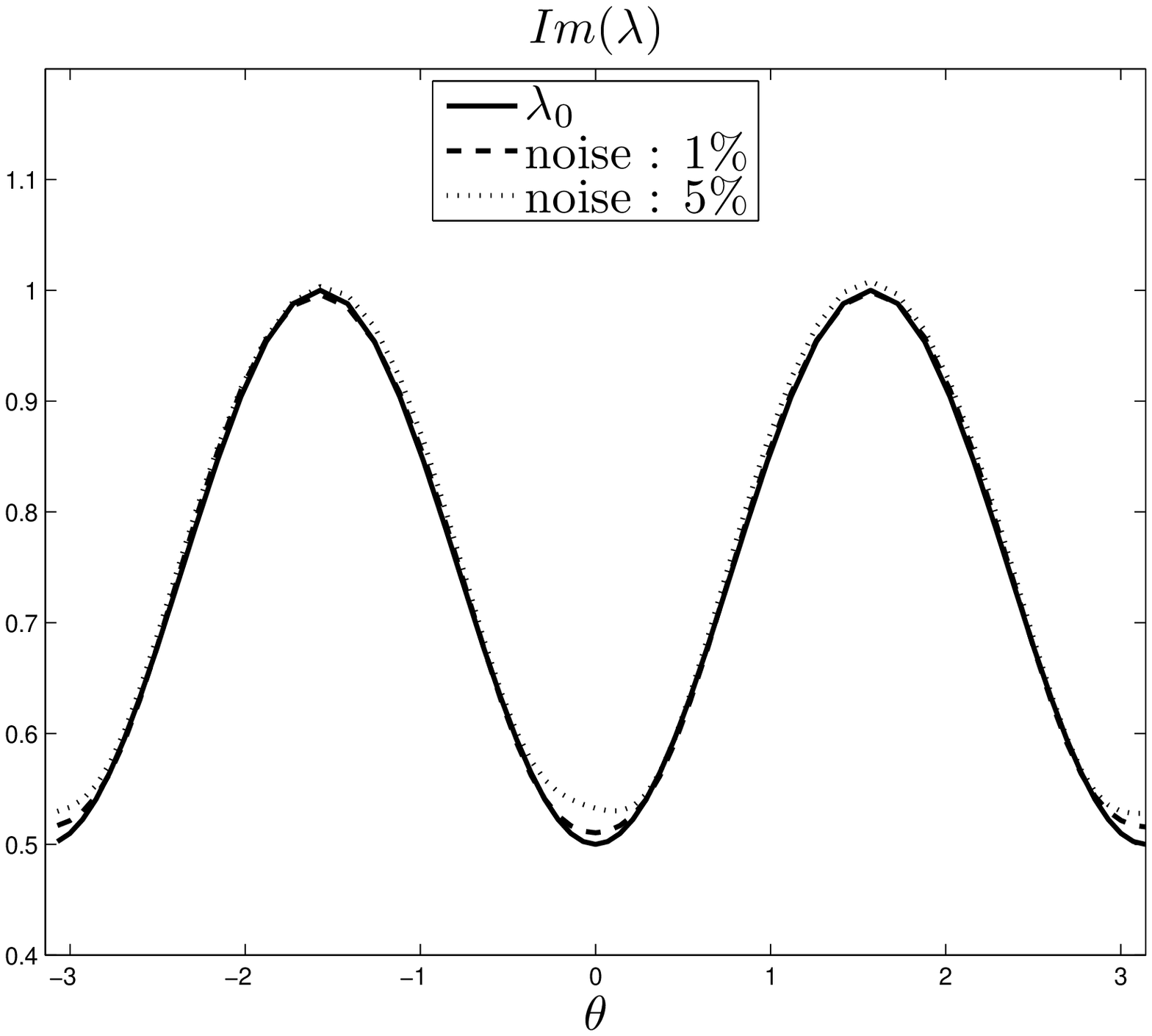}}
\subfigure{ \includegraphics[width =.49\textwidth]{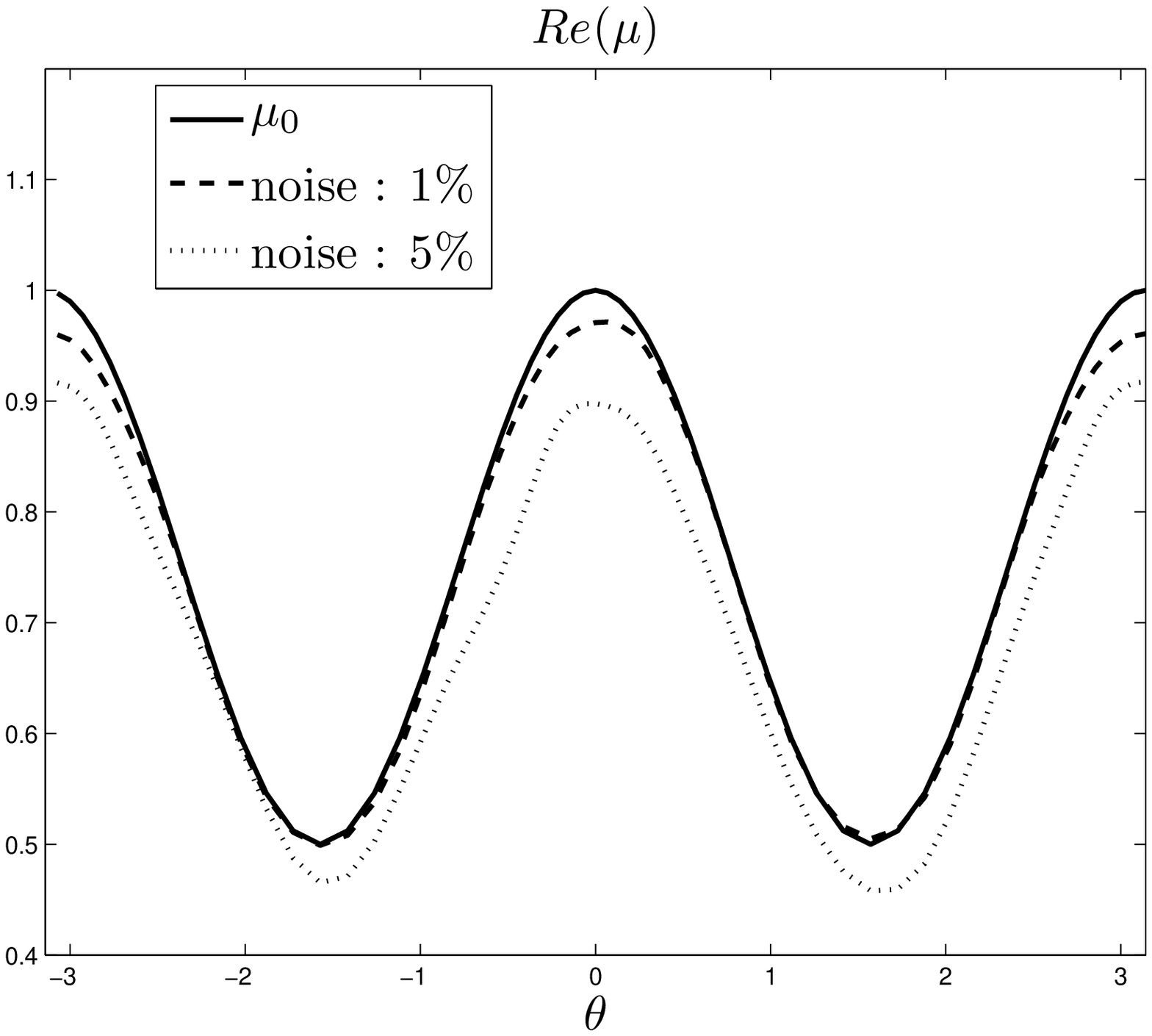}}
\caption{Wave number $k=9$, limited aperture data with $10$ incident waves, $ \mbox{Error }(1\%) =2.7 \%\,,\; \mbox{Error }(5\%) =9.8 \%$ reconstruction of $\Im m(\ld_0)=0.5(1+\sin^2(\theta))$, $\ld_{init} = 0.7i$ (on the left) and $\Re e(\mu_0)=0.5(1+\cos^2(\theta))$, $\mu_{init} = 0.7$ (on the right). }
\label{FigSim1}
\end{minipage}
\end{figure}
On Figure \ref{FigSim1} we show the reconstruction of functional impedance coefficients $\ld$ and $\mu$. The reconstruction is quite good for $1\%$ of noise and remains acceptable for $5\%$ of noise. 
\paragraph*{Stability with respect to a perturbed geometry.}
To illustrate the stability result with respect to the obstacle stated in Theorem \ref{ThStabObs} we construct numerically $u^{\infty}_{obs}(\cdot,d_j) = T(\ld_0,\mu_0,\pD,d_j)$ for a given obstacle $D$ and we minimize the ''perturbed'' cost function
\[
 F_{\ve}(\ld,\mu) =\frac{1}{2}\sum_{j=0}^N \Vert T(\ld,\mu,\pDe,d_j) - u^{\infty}_{obs}(\cdot,d_j) \Vert^2_{L^2(S_j)}
\]
for a perturbed obstacle $D_{\ve}$ in order to retrieve $(\ld_0,\mu_0)$ on $\pDe$.  
We first consider a perturbed obstacle $D_{\ve}$ which is our previous ellipse of semi-axis 0.4 and 0.3, the exact obstacle $D$ being
such ellipse once perturbed with an oscillation of amplitude $\gamma$. More precisely, the obstacle $D$ is parametrized by
\[x(t)=0.4 \left(\cos t + \gamma \cos(20 t)\right),\quad y(t)=0.3\left(\sin t + \gamma \sin(20 t)\right),\quad t \in [0,2\pi].\] 
In the following experiments we want to evaluate the impact of $\gamma$ on the reconstruction of the coefficients. 
The corresponding results are represented on Figure \ref{FigPertSim0} for $\ld=0$ and piecewise constant $\mu$, for two amplitudes $\gamma$ of perturbation.
\begin{figure}[h]
\begin{minipage}{\textwidth}
\subfigure[$ \mbox{Error }(1\%) =3.2 \%\,,\; \mbox{Error }(5\%) =10.8 \%$]{ \includegraphics[width =.49\textwidth]{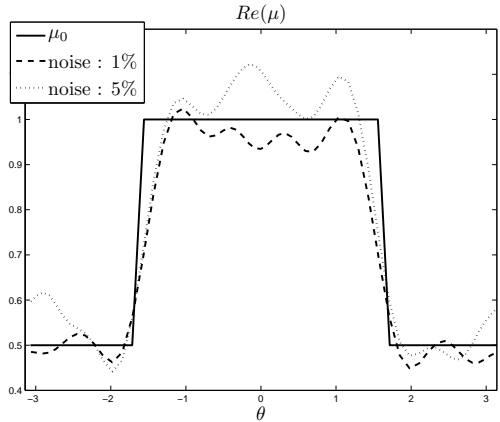}}
 \subfigure[$ \mbox{Error }(1\%) =11 \%\,,\; \mbox{Error }(5\%) =15 \%$]{\includegraphics[width =.49\textwidth]{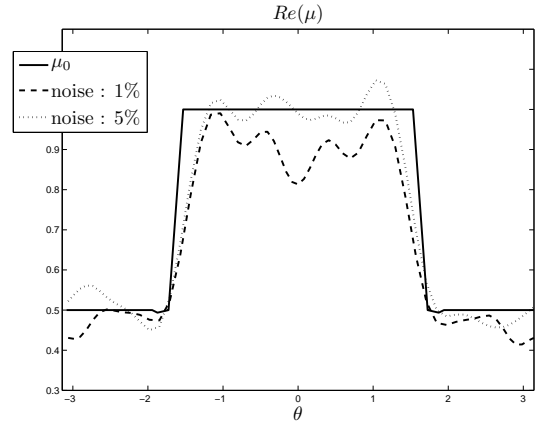}}
\caption{Perturbed ellipse, wave number $k=9$, $\mu_{init} = 0.7$, $\ld =0$, $\gamma=1\%$ on the left and $3\%$ on the right. }
\label{FigPertSim0}
\end{minipage}
\end{figure}

We now consider a second kind of perturbed obstacle, as indicated on Figure \ref{FigGeoPert}.
The perturbation of the obstacle is again denoted $\gamma$ and defined by 
\[
\gamma:=\frac{\ve_0}{\mbox{diam}(D)}. 
\]
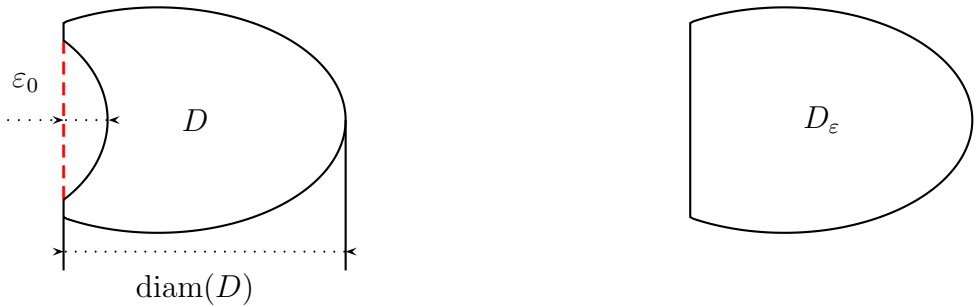
\begin{figure}[h!]
\begin{minipage}{.5\textwidth}
\centering
\begin{pspicture}(-2.5 ,-2.5)(5 ,2)
\psset{xunit=5cm} % on fixe les unit�s des axes.
\psset{yunit=5cm} % le coin inf�rieur gauche de l'image a donc pour coordonn�es (-2.5,2)
\rput(-0.35,0.1){\large $\ve_0$}
\rput(0.1,0){\large $D$}
\psline[linecolor=red,linestyle=dashed,linewidth=1pt](-0.25,-0.21)(-0.25,0.21)
\psline[linestyle=dotted]{>-<}(-0.28,0)(-0.1,0)
\psline[linestyle=dotted](-0.28,0)(-0.4,0)
\psline(-0.25,-0.26)(-0.25,-0.4)
\psline(0.5,0)(0.5,-0.4)
\psline[linestyle=dotted]{>-<}(-0.28,-0.35)(0.53,-0.35)
\rput(0.1,-0.45){\large diam($D$)}
\fileplot{Images/geo.dat}
\end{pspicture}
\end{minipage}
\begin{minipage}{.5\textwidth}
\centering
\begin{pspicture}(-2.5 ,-2.5)(5 ,2)
\psset{xunit=5cm} % on fixe les unit�s des axes.
\psset{yunit=5cm} % le coin inf�rieur gauche de l'image a donc pour coordonn�es (-2.5,2)
\rput(0.1,0){\large $D_{\ve}$}
\fileplot{Images/geopert.dat}
\end{pspicture}
\end{minipage}
\caption{Exact (on the left) and perturbed (on the right) geometries.}
\label{FigGeoPert}
\end{figure}
Note that the perturbed obstacle is the convex hull of the non--convex obstacle $D$. To satisfy the assumptions of Theorem \ref{ThStabObs} we have to check that we can find some $\ld_{\ve}$ and $\mu_{\ve}$ such that
\[
F_{\ve}(\ld_{\ve},\mu_{\ve}) \leq \delta \,
\]
for a small $\delta$. If we take the same uniformly distributed incident directions with $N=10$ and  $k=9$ as before (the wavelength is more or less equal to the diameter of $D$), we have 
\[
\sqrt{\frac{F_{\ve}(\ld_0,\mu_0)}{ \sum_{j=0}^N \Vert u^{\infty}_{obs}(\cdot,d_j) \Vert^2_{L^2(S_j)}}} = 
\left\{
\begin{array}{ll}
8\% &\quad \mbox{if  } \gamma=1\% \\
27\% &\quad \mbox{if  } \gamma=3\%. \\
\end{array}
\right.
\]
 These levels of perturbation on the cost function are too high to hope a good reconstruction of the coefficients. It is reasonable to consider we do not enlighten the obstacle in the direction of the non--convexity (since we have poor knowledge of such area). Let us suppose for example that we still have $10$ incident waves but now the incident directions belong to $[-\pi/2,\pi/2] $. We have the following relative errors with the actual impedances
\[
\sqrt{\frac{F_{\ve}(\ld_0,\mu_0)}{ \sum_{j=0}^N \Vert u^{\infty}_{obs}(\cdot,d_j) \Vert^2_{L^2(S_j)}}} = 
\left\{
\begin{array}{ll}
3\% &\quad \mbox{if  } \gamma=1\% \\
9\% &\quad \mbox{if  } \gamma=3\%. \\
\end{array}
\right.
\]
In this case we hope a good reconstruction of the impedance coefficients at least in the directions of incidence. 
The corresponding results are represented for $\ld=0$ and a smooth function $\mu$ on Figure \ref{FigPertSim}, for $\ld=0$ and piecewise constant $\mu$ on Figure \ref{FigPertSim_bis}, for two amplitudes $\gamma$ of perturbation.
We can see that reconstruction is good for $\gamma=1\%$ even if we put $\sigma=5\%$ of noise on the measurements. For $\gamma=3\%$  the reconstruction remains quite good in the non--perturbed area and acceptable in the perturbed area.
\begin{figure}[h]
\begin{minipage}{\textwidth}
\subfigure[$ \mbox{Error }(1\%) =2.8 \%\,,\; \mbox{Error }(5\%) =11 \%$]{ \includegraphics[width =.49\textwidth]{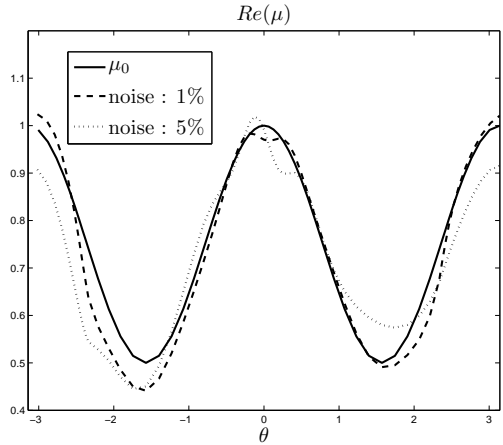}}
 \subfigure[$ \mbox{Error }(1\%) =8 \%\,,\; \mbox{Error }(5\%) =10.8 \%$]{\includegraphics[width =.49\textwidth]{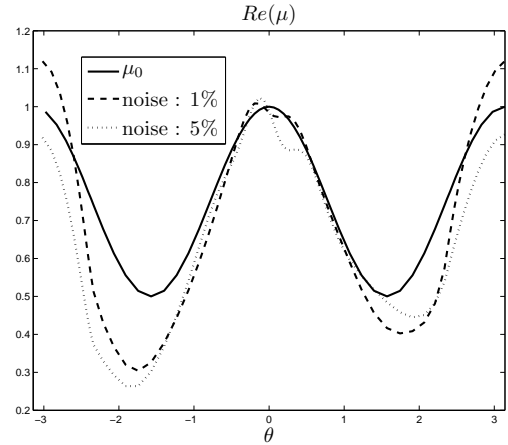}}
\caption{Perturbed obstacle (see Figure \ref{FigGeoPert}) and smooth function $\mu$, wave number $k=9$, $\mu_{init} = 0.7$, $\ld =0$, $\gamma=1\%$ on the left and $3\%$ on the right. }
\label{FigPertSim}
\end{minipage}
\end{figure}
\begin{figure}[h]
\begin{minipage}{\textwidth}
\subfigure[$ \mbox{Error }(1\%) =2.8 \%\,,\; \mbox{Error }(5\%) =11 \%$]{ \includegraphics[width =.49\textwidth]{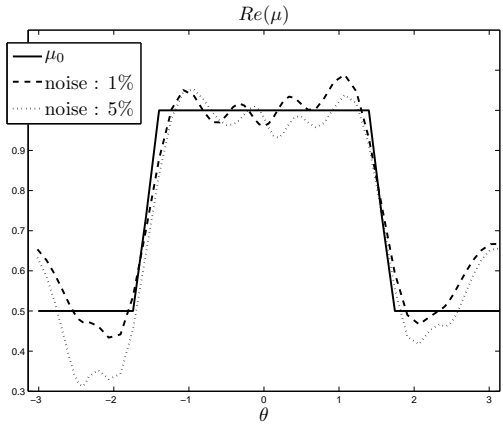}}
 \subfigure[$ \mbox{Error }(1\%) =2.9 \%\,,\; \mbox{Error }(5\%) =10.8 \%$]{\includegraphics[width =.49\textwidth]{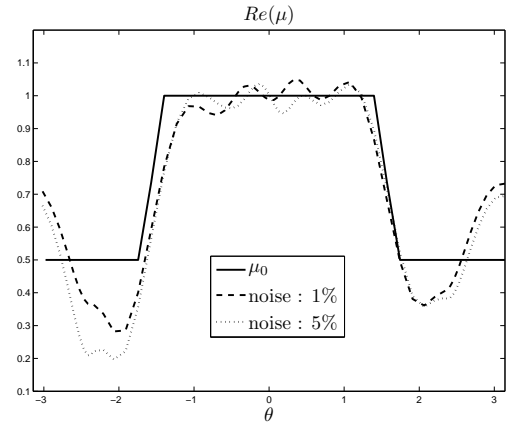}}
\caption{Perturbed obstacle (see Figure \ref{FigGeoPert}) and piecewise function $\mu$, wave number $k=9$, $\mu_{init} = 0.7$, $\ld =0$, $\gamma=1\%$ on the left and $3\%$ on the right. }
\label{FigPertSim_bis}
\end{minipage}
\end{figure}
\section*{Acknowledgement}
The work of Nicolas Chaulet is supported by a DGA grant.
\section*{References}
\bibliographystyle{plain}
\bibliography{bourgeois_chaulet_haddar_rev2}
\end{document}